\documentclass[12pt,norsk]{amsart}
\usepackage[T1]{fontenc}
\usepackage[utf8]{inputenc}
\usepackage{amsthm,amsfonts,amsmath,mathrsfs,amssymb}
\usepackage{dsfont}
\usepackage{fourier}
\usepackage{times,a4wide}

\newtheorem{theorem}{Theorem}[section]
\newtheorem*{theoremA}{Theorem A}
\newtheorem*{theoremB}{Theorem B}
\newtheorem{lemma}{Lemma}[section]
\newtheorem{coro}{Corollary}[section]

\newcommand{\Real}{\operatorname{Re}}
\newcommand{\Imag}{\operatorname{Im}}

\def\ch{\chi}

\newcommand{\T}{\mathbb T}

\renewcommand{\Re}{\mbox{Re}}

\newenvironment{proof*}{\textbf{\emph{Proof of }}}{\hspace*{\fill} $\square$}
\numberwithin{equation}{section}
\author[Kristian Seip]{Kristian Seip}
\address{Department of Mathematical Sciences \\ Norwegian University of Science and Technology \\ NO-7491 Trondheim \\ Norway}
\email{kristian.seip@ntnu.no}

\thanks{Research supported in part by Grant 275113 of the Research Council of Norway.}

\begin{document}

\subjclass[2010]{11M41, 11B37, 30K10}
\title[Universality and distribution of zeros and poles of some zeta functions]{Universality and distribution of zeros and poles \\ of some zeta functions  } 

\maketitle
 
 \begin{abstract}

This paper studies zeta functions of the form $\sum_{n=1}^{\infty} \chi(n) n^{-s}$, with $\chi$ a completely multiplicative function taking only unimodular values. We denote by $\sigma(\chi)$ the infimum of those $\alpha$ such that the Dirichlet series $\sum_{n=1}^{\infty} \chi(n) n^{-s}$ can be continued meromorphically to the half-plane $\Real s>\alpha$, and denote by $\zeta_{\chi}(s)$ the corresponding meromorphic function in $\Real s>\sigma(\chi)$.  We construct $\zeta_{\chi}(s)$ that have $\sigma(\chi)\le 1/2$ and are universal for zero-free analytic functions on the half-critical strip $1/2<\Real s <1$, with zeros and poles at any discrete multisets lying in a strip to the right of $\Real s =1/2$ and satisfying a density condition that is somewhat stricter than the density hypothesis for the zeros of the Riemann zeta function. On a conceivable version of Cram\'{e}r's conjecture for gaps between primes, the density condition can be relaxed, and zeros and poles can also be placed at $\beta+i \gamma$ with 
$\beta\le 1-\lambda \log\log |\gamma|/\log |\gamma|$  when $\lambda>1$. Finally, we show that there exists $\zeta_{\chi}(s)$ with $\sigma(\chi) \le 1/2$ and zeros at any discrete multiset in the strip $1/2<\Real s \le 39/40$ with no accumulation point in $\Real s >1/2$; on the Riemann hypothesis, this strip may be replaced by the half-critical strip $1/2 < \Real s < 1$.
 \end{abstract}

\section{Introduction}
\subsection{Background} This paper centres around Bohr's approach  
to the Riemann hypothesis, 
originating in his discovery \cite{B} that in any sub-strip of
$1/2<\Real s <1$, the set of points $s$ at which the Riemann zeta function $\zeta(s)$ takes the value $a$ for a fixed complex number $a\neq 0$, has positive lower density. In view of the Bohr--Landau theorem~\cite{BL} on the density of the zeros of $\zeta(s)$ to the right of the critical line, this cannot be true for $a=0$. Hence, as concluded by Titchmarsh in \cite[Ch. 11]{T}  , ``... the value $0$ of $\zeta(s)$, if it occurs at all in $\sigma>1/2$, is at any rate quite exceptional, zeros being infinitely rarer than $a$-values for any value of $a$ other than zero.'' It seems that this state of affairs led Bohr and others to believe in the unlikeliness of such ``exceptional'' zeros and that the Riemann hypothesis could be proved by establishing that $\zeta(s)$ is quasi-periodic in an appropriate sense in the strip $1/2<\Real s <1$. While the Riemann hypothesis is indeed equivalent to an assertion about quasi-periodicity, as proved by Bagchi \cite{Ba} (see Theorem~B below), our aim is to show that there exist zeta functions with zeros located essentially anywhere in a strip to the right of $\Real s=1/2$, subject to a density restriction akin to the density hypothesis for the zeros of the Riemann zeta function, and whose value distribution properties otherwise cannot be easily distinguished from those of $\zeta(s)$.

The zeta functions that we will consider, are of the form 
\begin{equation} \label{eq:defzeta} \zeta_{\chi}(s):=\sum_{n=1}^{\infty} \chi(n) n^{-s}=\prod_{p} \frac{1}{\left(1-\chi(p) p^{-s}\right)}, \end{equation} 
where $\chi$ is a completely multiplicative function taking only unimodular values and the product to the right is over the sequence of prime numbers $p$. This definition and the equality to the right make sense for $\sigma:=\Real s>1$, as the abscissa of absolute convergence is $1$ for both the Dirichlet series and the Euler product in \eqref{eq:defzeta}. We let $\sigma(\chi)$ denote the infimum of those $\alpha$ such that the function defined by \eqref{eq:defzeta} can be continued meromorphically to the half-plane $\Real s>\alpha$. We use the same symbol $\zeta_{\chi}(s)$ for the meromorphic extension of $\sum_{n=1}^{\infty} \chi(n) n^{-s}$ to $\Real s>\sigma(\chi)$
(or to $\Real s\ge \sigma(\chi)$ if this makes sense) and declare it to be the Helson zeta function associated with~$\chi$. 
Our usage of the symbol $\chi$ comes from the identification of these functions as characters, as they constitute the (compact) dual group of the discrete multiplicative group of positive rationals $\langle\mathbb{Q}_+,\cdot\rangle$. This character group is closed also in the following analytic sense: The functions $\zeta_{\chi}(s)$ are precisely the vertical limit functions of $\zeta(s)$ in the half-plane $\sigma>1$, i.e., those functions that are obtained as limits of sequences of vertical translates $\zeta(s+i\tau)$,  with $\tau$ in $\mathbb{R}$. We refer to 
\cite{HLS} for more information about these points.

Strictly speaking, $\zeta(s)$ itself is the only Helson zeta function among all Dirichlet $L$-functions.  We may however think of any such $L$-function as a Helson zeta function since it can be made into a function of the form $\zeta_{\chi}(s)$ by multiplication by a finite Euler product, and such a transformation does not change the basic analytic properties of the $L$-function in the half-plane $\sigma>0$. For the same reason, we will have $\sigma(\chi)=-\infty$ for the corresponding characters $\chi$. 

We will sometimes think of the numbers $\chi(p)$---or a subsequence of these numbers---as a sequence of independent Steinhaus random variables.  According to this model, 
$\chi$ itself or a subsequence of the numbers $\chi(p)$ may be considered as a point on the infinite-dimensional torus $\T^{\infty}$, equipped with the natural  product probability measure. This measure is obtained as a product of normalized arc length measure on the unit circle for each of the variables $\chi(p)$. We have chosen to use the term ``Helson zeta function'' because Helson observed in \cite{H} that, almost surely, the Dirichlet series of $\zeta_{\chi}(s)$ converges and has no zeros in $\sigma>1/2$, whence in particular $\sigma(\chi)\le 1/2$ (see  also \cite[Cor. 4.7]{HLS}). This random model is commonly used in the study of statistical properties of $\zeta(s)$. We refer to the paper by Saksman and Webb \cite{SW}, where it was shown that in fact $\sigma(\chi)=1/2$ holds almost surely.

Helson's observation reflects in a rather compelling way the important point that the multiplicative structure of $\zeta_{\chi}(s)$, combined with moderate growth in the vertical direction, so to speak ``forces'' the zeros and poles of $\zeta_{\chi}(s)$ in $\sigma>1/2$, if any, towards the critical line. Familiar arguments in the theory of the Riemann zeta function allow us to establish quantitative results to this effect, for example a variant of the Bohr--Landau theorem, on the proviso that $\zeta_{\chi}(s)$ grows at most polynomially in the vertical direction.  One should bear in mind that there is no symmetry about the critical line in this respect, even when $\sigma(\chi)<1/2$. To see this, it suffices to choose $\chi(n)$ to be the Liouville function, i.e., $\chi(p)=-1$ for all $p$, so that $\zeta_{\chi}(s)=\zeta(2s)/\zeta(s)$. In this case, the poles and zeros of $\zeta_{\chi}(s)$ in the critical strip are expected to lie respectively on $\sigma=1/2$ and $\sigma=1/4$. 

The picture is, however, strikingly different in the general case, as we will see from Theorem~\ref{thm:nonuniversal} below: The geometry of the zeros of $\zeta_{\chi}(s)$ can, in the literal sense, be completely arbitrary in the strip $1/2<\Real s < 1$, at least if we assume the Riemann hypothesis to be true. Hence, with no a priori restriction on its meromorphic extension, the value distribution of $\zeta_{\chi}(s)$ may be rather more complex than and dramatically different from that of the Riemann zeta function.

Our analysis of Helson zeta functions will rely on an extension of the Voronin universality theorem \cite{V}, which is the most remarkable result in the line of research initiated by Bohr on the value distribution of $\zeta(s)$. We will state Voronin's theorem as it was developed in subsequent work of Reich \cite{R} and Bagchi \cite{Ba}. To this end, let $\Omega$ denote the strip $1/2<\Real s < 1$ and $H(\Omega)$ be the space of analytic functions on $\Omega$, equipped with the natural topology of locally uniform convergence; we let $H^*(\Omega)$ be the subset of those $h(s)$ in $H(\Omega)$ such that also $1/h(s)$ is in $H(\Omega)$. Moreover, we let $M(\Omega)$ be the larger space of meromorphic functions on $\Omega$, for which we use the topology of locally uniform convergence in the spherical metric.

Recall that the lower and upper density of a measurable set of positive real numbers $A$ are defined respectively as
\[ \underline{d}(A):=\liminf_{T\to\infty} \frac{m\left(A\cap \{ t: 0<t\le T\}\right)}{T} \quad \text{and} \quad 
\overline{d}(A):=\limsup_{T\to\infty} \frac{m\left(A\cap \{ t: 0<t\le T\}\right)}{T},\]
where $m$ denotes Lebesgue measure on the real line. If $\underline{d}(A)=\overline{d}(A)$, then $A$ has a density $d(A)$ which is this common value.
We say that a function $h(s)$ in $M(\Omega)$ is universal for $H^*(\Omega)$ if, given any $f(s)$ in $H^*(\Omega)$, $\varepsilon>0$, and compact subset $K$ of  $\Omega$, 
\[ \underline{d}\left(\left\{t: \max_{s\in K} |\ h(s+it)-f(s)|<\varepsilon \right\} \right) >0.\]
Bagchi's version of the universality theorem reads as follows \cite[Thm. 3.1]{Ba}.
\begin{theoremA} Every Dirichlet $L$-function is universal for $H^*(\Omega)$.
\end{theoremA}
In fact, an even stronger result concerning joint universality of the $L$-functions associated with the Dirichlet characters to a given modulus $k$ was proved in \cite{Ba}. In Section~\ref{sec:univ}, we will establish a general condition for Voronin universality  showing in particular that, almost surely, $\zeta_{\chi}(s)$ is universal for $H^*(\Omega)$. 

By a slight extension of  Bagchi's notion of strong recurrence \cite{Ba}, we say that a function $h(s)$ in $M(\Omega)$ is a strongly recurrent point for vertical translations if for every compact subset $K$ of $\Omega$ and $\varepsilon>0$,
\begin{equation}\label{eq:lowup} \overline{d}\left(\left\{\tau: \max_{s\in K}\frac{|h(s+i\tau)-h(s)|}{(1+|h(s+i\tau)|)(1+|h(s)|)}<\varepsilon \right\} \right) >0.\end{equation}
In \cite[Thm. 3.7]{Ba}, Bagchi used Theorem~A and the Bohr--Landau theorem for Dirichlet $L$-functions to establish the following equivalence\footnote{It may be observed that Theorem~B would remain true if we in \eqref{eq:lowup} had used the lower density instead of the upper density to define ``strong recurrence''.} for the generalized Riemann hypothesis.

\begin{theoremB}
A Dirichlet $L$-function is zero-free in the half-plane $\Real s>1/2$ if and only if it is a strongly recurrent point for vertical translations in $M(\Omega)$. 
\end{theoremB}
In Section~\ref{sec:univ}, we will observe that this theorem extends as well to a wide class of Helson zeta functions, and we may in particular conclude that $\zeta_{\chi}(s)$ is almost surely a strongly recurrent point for vertical translations in $M(\Omega)$. 

\subsection{Statement of main results} 
The Bohr--Landau theorem\footnote{A bound of the form $O(T^{1-\varepsilon})$ was first established by Carlson \cite{Ca}, by a refinement of the work of Bohr and Landau \cite{BL} which only established the bound $o(T)$.} asserts that 
\begin{equation}  \label{eq:BLthm} N(\sigma, T)=O\left(T^{1-\varepsilon}\right) \end{equation}
for some $\varepsilon=\varepsilon(\sigma)$, $0<\varepsilon<1$, whenever $\sigma>1/2$, where as usual $N(\sigma,T)$ denotes the number of zeros $\rho=\beta+i\gamma$ of $\zeta(s)$ satisfying $\beta>\sigma$ and $0<\gamma \le T$. It is clear that a similar sparseness condition for the zeros and the poles of $\zeta_{\chi}(s)$ must be required for $\zeta_{\chi}(s)$ to be universal for $H^*(\Omega)$. A slight adjustment of the conclusion \eqref{eq:BLthm} of the Bohr--Landau theorem, called the ``Bohr--Landau condition'', will therefore play a pivotal role in our treatment of universality. 
 
On the other hand, we may ask whether any sequence satisfying a condition similar to \eqref{eq:BLthm} may constitute the zeros of a Helson zeta function that is universal for $H^*(\Omega)$. In fact, keeping in mind that $\zeta_{-\chi}(s)=\zeta_{\chi}(2s)/\zeta_{\chi}(s)$ and hence that the zeros of $\zeta_{\chi}(s)$ coincide with the poles of $\zeta_{-\chi}(s)$ in $\Real s>1/2$, and vice versa, we may ask the more general question of whether zeros and poles can be placed anywhere, subject to a sparseness condition like \eqref{eq:BLthm}. Our first theorem gives essentially an affirmative answer on the proviso that the points stay sufficiently close to the critical line.

We will use the following terminology. We say that a set of points $\mathcal{Z}$ in the complex plane is a signed multiset if there is a multiplicity $m_{\mathcal{Z}}(\rho)$ in $\mathbb{Z}\setminus \{0\}$ associated with every $\rho$ in $\mathcal{Z}$. We may declare that $m_{\mathcal{Z}}(s):=0$ if $s$ is not in $\mathcal{Z}$ and define the sum $\mathcal{Z}+\mathcal{Y}$ of two signed multisets $\mathcal{Z}$ and $\mathcal{Y}$ to be the set of numbers $\rho$ such that 
$m_{\mathcal{Z}+\mathcal{Y}}(\rho):=m_{\mathcal{Z}}(\rho)+m_{\mathcal{Y}}(\rho)\neq 0$. The set of zeros and poles of a meromorphic function $h(s)$ constitutes in an obvious way a signed multiset, which we will denote by $Z(h(s))$. We observe that $Z(h(s)g(s))=Z(h(s))+Z(g(s))$ if two meromorphic functions $h(s)$ and $g(s)$ are defined on the same domain. 
A signed multiset $\mathcal{Z}$ is said to be a multiset if $m_\mathcal{Z}(\rho)>0$ for all $\rho$ in $\mathcal{Z}$. We write $Z^+(h(s))$ for the multiset of zeros of the meromorphic function $h(s)$.
We will frequently refer to (signed) multisets as ordinary sets, without explicit reference to the associated multiplicity of its elements. In particular, we will permit ourselves to think of (signed) multisets as subsets of ordinary sets and to perform intersections with ordinary sets.

When the signed multiset $\mathcal{Z}$ is a subset of a domain $\Delta$, we say that it is locally finite in $\Delta$ if there are only finitely many points $\rho$ from $\mathcal{Z}$ in each compact subset of $\Delta$. For a locally finite signed multiset~$\mathcal{Z}$ in $\Real s > 1/2$, we have the following natural analogue of the counting function in \eqref{eq:BLthm}: 
\[ N_{\mathcal{Z}}(\sigma,T):= \sum_{\rho=\beta+i\gamma \in \mathcal{Z}: \beta>\sigma,  |\gamma|\le T} |m_{\mathcal{Z}}(\rho)| \]
for $\sigma\ge 1/2$.
The Bohr--Landau condition will simply be that $N_{\mathcal{Z}}(\sigma,T)=o(T)$ for every $\sigma>1/2$, which is a slight weakening of \eqref{eq:BLthm}.
Our condition for $N_{\mathcal{Z}}(\sigma,T)$ will depend on our knowledge of large prime gaps. Unconditionally, thanks to 
a theorem of Baker, Harman, and Pintz \cite{BHP}, we know that 
\begin{equation} \label{eq:primegap} \pi(x+\xi) - \pi(x)\gg \xi/\log x \end{equation}
holds for large $x$ when $\xi=x^{21/40}$, where as usual $\pi(x)$ is the number of primes not exceeding $x$.   

We are now ready to state the unconditional version of our first main result. In view of the Bohr--Landau condition, we observe that condition (c) below is essentially optimal close to $\sigma=1/2$. 

\begin{theorem}\label{thm:main}
Let $\mathcal{Z}$ be any locally finite signed multiset in the half-plane $\Real s > 1/2 $ such that 
\begin{itemize}
\item[(a)] $\mathcal{Z}$ is a subset of the strip $1/2 < \Real s \le \alpha$ for some $1/2<\alpha<59/80$. 
\end{itemize}
Suppose also that $\mathcal{Z}$ satisfies the following conditions:
\begin{itemize} 
\item[(b)] $N_{\mathcal Z}(\sigma,T+1)-N_{\mathcal Z}(\sigma, T)=O(T^{\varepsilon})$ for every $\sigma>1/2$ and $\varepsilon>0$;
\item[(c)] $N_{\mathcal{Z}}(\sigma,T)=O\big(T^{\frac{\alpha-\sigma}{\alpha+\sigma-1}}\big)$. 
\end{itemize}  
 Then there exists a Helson zeta function $\zeta_{\chi}(s)$ with $\sigma(\chi)\le 1/2$ so that 
\begin{itemize} 
\item[(i)] $\mathcal{Z}$ is the set of zeros and poles of $\zeta_{\chi}(s)$ in $\Real s >1/2$;
\item[(ii)] $\zeta_{\chi}(s)$ is universal for $H^*(\Omega)$;
\item[(iii)] $\zeta_{\chi}(s)$ is not a strongly recurrent point for vertical translations in $M(\Omega)$. 
\end{itemize}   
\end{theorem}

On the Riemann hypothesis, \eqref{eq:primegap} would still hold for $\xi=c x^{1/2} \log x$ and a suitable constant $c$. This would allow us to replace the fraction  $59/80$  in (a) by $3/4$. For our purposes, however, it would suffice to know a little less than \eqref{eq:primegap}, for example that, say,
\begin{equation} \label{eq:primegapc} \pi(x+\xi)-\pi(x)\ge  \frac{\xi}{(\log x)^{2+\varepsilon}} \end{equation}
for some $\varepsilon>0$ and suitable $\xi$, depending on $x$. A well known conjecture of Cram\'{e}r in the distribution of prime numbers \cite{Cr2, Gr}, based on his famous random model, asserts that $G(x)=O\big((\log x)^2\big)$, where $G(x)$ is the distance from $x$ to the smallest prime larger than $x$. While some doubt has been cast on this conjecture \cite{P2}, it seems still conceivable, as hinted at by Pintz in \cite{P1}, that $G(x)=O\big((\log x)^{2+\varepsilon}\big)$ may hold for every $\varepsilon>0$. If this were true, then the  upper bound $59/80$  in (a) would be increased to the optimal value $1$. We could then choose $\alpha$ arbitrarily close to $1$. 

We notice at this point that our density condition (c) is similar to the density hypothesis for the zeros of $\zeta(s)$, i.e. the famous unproven assertion
\[ N(\sigma,T)=\left(T^{1-2(\sigma-1/2)+\varepsilon} \right) \]
that arose from Ingham's work \cite{In}. However, even with $\alpha$ arbitrarily close to $1$, our condition is still weaker than the density hypothesis. 
It would be interesting to sharpen (c) so that our density condition would be ``in accordance'' with the density hypothesis on the assumption that $\alpha$ can be chosen arbitrarily close to $1$. 

Our next theorem is a conditional variant of Theorem~\ref{thm:main} in which we address what happens if we go one step further and allow the points of $\mathcal{Z}$ to approach the $1$-line:

\begin{theorem}\label{thm:main3}
Assume that $G(x)=O\big((\log x)^{2+\varepsilon}\big)$ for every $\varepsilon>0$. Let $\mathcal{Z}$ be any locally finite signed multiset in the half-plane $\Real s > 1/2 $ such that  for some $\lambda>\kappa>1$ 
\begin{itemize}
\item[(a)] every $\sigma=\beta+i\gamma$ in $\mathcal{Z}$ satisfies $|\gamma|\ge e^e$ and
$ 1/2<\beta \le 1-\lambda \frac{\log\log |\gamma|}{\log |\gamma|} $;
\item[(b)] $N_{\mathcal Z}(1/2,T)=O\left((\log T)^{\kappa-1}\right)$.
\end{itemize}  
 Then there exists a Helson zeta function $\zeta_{\chi}(s)$ with $\sigma(\chi)\le 1/2$ so that 
\begin{itemize} 
\item[(i)] $\mathcal{Z}$ is the set of zeros and poles of $\zeta_{\chi}(s)$ in $\Real s >1/2$;
\item[(ii)] $\zeta_{\chi}(s)$ is universal for $H^*(\Omega)$;
\item[(iii)] $\zeta_{\chi}(s)$ is not a strongly recurrent point for vertical translations in $M(\Omega)$. 
\end{itemize}   
\end{theorem}

Conclusion (iii) of either of the two theorems above may be strengthened if an additional restriction is put on $\mathcal{Z}$. To see this, we need the following terminology. In contrast to the notion of strong recurrence, we say that $h(s)$ in $M(\Omega)$ is a wandering point for vertical translations if there exist a compact subset $K$ of $\Omega$ and $\varepsilon>0$ such that 
\[ \max_{s\in K}\frac{|h(s+i\tau)-h(s)|}{(1+|h(s+i\tau)|)(1+|h(s)|)} \ge \varepsilon \]
for every sufficiently large $\tau$. We notice that if $h(s)$ has only a finite number of zeros and poles in any strip $1/2+\varepsilon \le \Real s \le 1-\varepsilon $, then $h(s)$ is a wandering point for vertical translations. We see this by choosing $K$ to be a closed disc centred at any of the zeros or poles of $h(s)$, with $K$ so small that $K\subset \Omega$ and there are no other zeros or poles in $K$.  For every sufficiently large translation parameter $\tau$, there is neither a zero nor a pole in $K+i\tau$, and for such $\tau$ we may use Rouch\'{e}'s theorem to conclude. 

Hence, in the special case when $\mathcal Z$ has finitely many points in  $1/2+\varepsilon \le \Real s\le 1-\varepsilon$ for every $\varepsilon$, $0<\varepsilon<1/4$, we may replace conclusion (iii) of both Theorem~\ref{thm:main} and Theorem~\ref{thm:main3} by the following stronger assertion: 
\begin{itemize}
\item[(iii')] $\zeta_{\chi}(s)$ \emph{is  a wandering point for vertical translations in} $M(\Omega)$.
\end{itemize}  

Theorem~\ref{thm:main} requires the zeros and poles to be at a positive distance to the $1$-line. We have not been able to improve the upper bound $\alpha < 59/80$ unconditionally, but curiously, via $\zeta(s)$ itself, we are indeed able to place a pole at the ``extreme'' point $s=1$:  

\begin{theorem}\label{thm:main2}
Let $\alpha<\nu$ be two numbers in the interval $(1/2,59/80]$ and $\mathcal{Z}$ be any locally finite signed multiset in the half-plane $\Real s > 1/2 $ such that 
\begin{itemize}
\item[(a)] $m_{\mathcal{Z}}(\nu)>0$ and $\mathcal{Z}\setminus \{\nu\}$ is a subset of the strip $1/2 < \Real s < \alpha$ .
\end{itemize}
Suppose also that $\mathcal{Z}$ satisfies the following conditions:
\begin{itemize} 
\item[(b)] $N_{\mathcal Z}(\sigma,T+1)-N_{\mathcal Z}(\sigma, T)=O(T^{\varepsilon})$ for every $\sigma>1/2$ and $\varepsilon>0$;
\item[(c)] $N_{\mathcal{Z}}(\sigma,T)=O\big(T^{\frac{\alpha-\sigma}{\alpha+\sigma-1}}\big)$. 
\end{itemize}  
 Then there exists a Helson zeta function $\zeta_{\chi}(s)$ with $\sigma(\chi)\le 1/2$ so that 
\begin{itemize} 
\item[(i)] the set of zeros and poles of $\zeta_{\chi}(s)$ in $\Real s>1/2$ is the restriction to this half-plane of the signed multiset $Z(\zeta(s))+\mathcal{Z}$;
\item[(ii)] $\zeta_{\chi}(s)$ is universal for $H^*(\Omega)$; 
\item[(iii)] $\zeta_{\chi}(s)$ is not a strongly recurrent point for vertical translations in $M(\Omega)$.
\end{itemize}   
\end{theorem}

We will see during the course of the proof that this result could be elaborated to allow meromorphic continuation as well as zeros and poles beyond the critical line. We have chosen the current version to have a statement that is suitably ``aligned'' with Theorem~\ref{thm:main} and has essentially the same proof. 

We could of course have stated a conditional version of Theorem~\ref{thm:main2}, assuming either the Riemann hypothesis or Cram\'{e}r's conjecture, but this would essentially just mean that $\nu$ could be placed closer to the $1$-line. The zero at $\nu$ prevents us from placing other zeros closer to the $1$-line, so that we are unable to obtain an analogue of Theorem~\ref{thm:main3}.  

A reasonable conclusion to be drawn from the three theorems stated above is that 
Theorem~B, while a striking reformulation of the generalized Riemann hypothesis, may be an unlikely first step in establishing the truth of it if no other characteristic feature of the Dirichlet $L$-functions than Voronin universality is taken into account. 

Our fourth theorem shows that a sparseness condition of Bohr--Landau-type is a rather drastic restriction.  

\vbox{\begin{theorem}\label{thm:nonuniversal}
Let $\mathcal{Z}^+$ be a locally finite multiset  in $\Real s > 1/2 $. Suppose that at least one of the following two conditions hold:
\begin{itemize}
\item[(i)] $\mathcal{Z}^+$ is a subset of the strip $1/2 < \Real s  \le 39/40$; 
 \item[(ii)] $\mathcal{Z}^+$ is a subset of $1/2<\Real s < 1$ and the Riemann hypothesis is true.   
\end{itemize}
Then there exists a Helson zeta function $\zeta_{\chi}(s)$ with $\sigma(\chi)\le 1/2$ so that 
 $\mathcal{Z}^+$ is the set of zeros of $\zeta_{\chi}(s)$ in $\Real s >1/2$. 
\end{theorem}}

Hence, in particular, there exist Helson zeta functions $\zeta_{\chi}(s)$ with $\sigma(\chi)\le 1/2$ that fail spectacularly to be universal for $H^*(\Omega)$.   
Here we have chosen to confine ourselves to the construction of $\zeta_{\chi}(s)$ with prescribed zeros, because this can be done with essentially the same method as that used to prove Theorem~\ref{thm:main}. The more general problem of constructing $\zeta_{\chi}(s)$ with prescribed zeros and poles, on the other hand, would require a further elaboration of our method which we have chosen not to pursue in this paper.

For every character $\chi$, there exists a sequence of vertical translates $\tau_n$ such that
\[ \zeta_{\chi}(s)=\lim_{n\to \infty} \zeta(s+i\tau_n), \]
with uniform convergence on compact subsets of the half-plane $\Real s>1$. Keeping this in mind, we may think of Theorem~\ref{thm:nonuniversal} as expressing another kind of universality of $\zeta(s)$: On the Riemann hypothesis, any conceivable set of zeros in the strip $1/2 < \Real s < 1$ for a function meromorphic in $\Real s>1/2$ can be reached via local uniform convergence of vertical translates of $\zeta(s)$ in $\Real s >1$, along with meromorphic continuation.

From another point of view, the appearance of the Riemann hypothesis in (ii) may perhaps seem a little deceptive, because Theorem~\ref{thm:nonuniversal} has essentially no relation to arithmetic. In fact, an analogous statement about ``universality'' of zeros in $\Omega$ could be made for Euler products with the functions $p^{-s}$ replaced by $\lambda_n^{-s}$ for any reasonably regular sequence $\lambda_n$ satisfying $\lambda_n\sim n\log n$. Curiously, the Riemann hypothesis implies exactly the regularity we need, expressed in terms of the admissible range $h\ge c \sqrt{x} \log x$ in \eqref{eq:primegap}, and this is why we have chosen the formulation of condition (ii) above.  
 
On the assumption that $G(x)=O\big((\log x)^{2+\varepsilon}\big)$ for every $\varepsilon>0$, we could prove an analogue of Theorem~\ref{thm:nonuniversal} for  $\mathcal Z^+$ being a subset of the entire critical strip $0<\Real s <1$. In this case, a minor extra precaution would have to be taken close to the $1$-line because of the extra logarithmic factor in \eqref{eq:primegapc} compared to \eqref{eq:primegap}. Our proof of Theorem~\ref{thm:nonuniversal} should make it clear how to proceed, and we will therefore refrain from entering the details of such a conditional construction.

\subsection{Outline of the paper} 
We begin in the next section by clarifying the following simple point: When $\zeta_{\chi}(s)$ has a meromorphic continuation across the $1$-line, the intersection of $Z(\zeta_{\chi}(s))$ with that line can consist of at most one point, and this point can only be a simple pole or a simple zero. This result is of some basic importance and will have several applications in subsequent sections.

In Section~\ref{sec:univ}, we turn to our condition for universality and our extensions of Theorem A and Theorem B. Our approach differs from previous work in this area (see for example \cite{St}) in that we focus on purely multiplicative conditions for universality of zeta functions. Indeed, our condition for universality rests on two pillars, one arithmetic and one analytic: Kronecker's approximation theorem and approximation of analytic functions by finite Euler products. For this reason, we work exclusively with $\log \zeta_{\chi}(s)$ rather than with $\zeta_{\chi}(s)$ itself. As in earlier work, bounded mean squares play a crucial role in carrying out the actual approximation of analytic functions, but now the mean squares are computed for $\log \zeta_{\chi}(s)$, or, to be more precise, we rely on the mean square distance from $\log \zeta_{\chi}(s)$ to the logarithm of finitely many factors of the Euler product of $\zeta_\chi(s)$. Convergence of this distance requires much less from $\zeta_{\chi}(s)$ than the boundedness of the mean squares of $\zeta_{\chi}(s)$. We need however to add the Bohr--Landau condition, which is not automatically implied by the mean square convergence of the logarithms of the finite Euler products.

The primary goal of Section~\ref{sec:finite} is to show that our ``multiplicative'' condition for universality, expressed in terms of $\log \zeta_{\chi}(s)$, implies the traditional ``additive'' condition, expressed in terms of $\zeta_{\chi}(s)$. From a function theoretic point of view, the distinction between the two conditions can be related to the classical notions of respectively functions of bounded type and functions of finite order, and our arguments rely on the canonical factorization of functions in either of these classes. Up to an inessential factor, a function of finite order is a bounded analytic function, while a function of bounded type is the ratio of two bounded analytic functions. We introduce and discuss these notions in the framework of Helson zeta functions and show in particular that the ``explicit formula'' for $\zeta'_{\chi}(s)/\zeta_{\chi}(s)$ becomes much more precise when $\zeta_{\chi}(s)$ is assumed to be of finite order rather than of bounded type. Nevertheless, digressing briefly from our main discussion, we are able to supply arguments to show that if $\zeta_{\chi}(s)$ extends to an analytic function of bounded type in a half-plane including the $1$-line and has a zero or a pole on that line, then $\zeta_{\chi}(s)$ has a zero-free region of the classical de la Vall\'{e}e--Poussin type whenever a natural density condition for the zeros is met.

In Section~\ref{sec:prelim}, we have collected some auxiliary results to be used in the proof of our main theorems. Here we express in precise terms the intuitive idea that we should make sense of 
\begin{equation} \label{eq:intuitive} \frac{\zeta'_{\chi}(s)}{\zeta_{\chi}(s)} - \sum_{\rho} \frac{m_{\mathcal{Z}}(\rho)}{(s-\rho)} \end{equation}
as an analytic function in $\Real s > 1/2$, when constructing $\zeta_{\chi}(s)$ with $Z\left(\zeta_{\chi}(s)\right)=\mathcal{Z}$. We need to modify \eqref{eq:intuitive} to get a manageable problem. First, the problem becomes easier if we replace $\zeta'_{\chi}(s)/\zeta_{\chi}(s)$ by a Dirichlet series over a carefully chosen subsequence of the primes and associated values for the character $\chi$. Then the remaining part of $\zeta'_{\chi}(s)/\zeta_{\chi}(s)$ can be found using our random model. Second, the sum over $\rho$ in \eqref{eq:intuitive} need not converge and even if it does, it may be hard to relate the sum to a Dirichlet series over prime powers. The solution to the latter problem will be to multiply each term in \eqref{eq:intuitive} by a suitable exponential factor, allowing us to write down manageable Mellin transforms. By our density condition on $\mathcal{Z}$, this can be done such that we also have absolute convergence of the sum in $\Real s >1/2$. The proofs in the two subsequent sections exhibit the details of such a construction. Section~\ref{sec:prelim} also contains some general estimates required to check  the mean square condition of our universality theorem (Theorem~\ref{thm:vor2}). 

The next three sections give the proofs of our main theorems. We begin in Section~\ref{sec:constr} with the the first step of the proof Theorem~\ref{thm:main}, which consists in picking a sub-product of the Euler product of $\zeta(s)$, extending to a meromorphic function with just one pole of the required multiplicity at $s=\nu$ and no other zeros or poles. When doing this ``surgery'' on the Euler product of $\zeta(s)$, we are faced with many of the same challenges that will appear in the main part of the proof. The situation is however simpler because the sum in \eqref{eq:intuitive} ``degenerates'' into a single term. 

The proofs of Theorem~\ref{thm:main}, Theorem~\ref{thm:main3}, and Theorem~\ref{thm:main2} are presented jointly in Section~\ref{thm:main}. The additional challenge in this section is to pick suitable exponential factors in the sum in \eqref{eq:intuitive}, as alluded to above. For the proof of Theorem~\ref{thm:main2}, it is essential that we use primes from the ``cutout'' Euler product from Section~\ref{sec:constr} to construct the corresponding Euler product.

In the final Section~\ref{sec:proofany}, we prove Theorem~\ref{thm:nonuniversal}. We rely on essentially the same construction as before, but resort in this case also to a special dyadic decomposition of the strip $1/2<\Real s < 1$  and a corresponding grouping of the points $\rho$ of the multiset $\mathcal{Z}^+$. In addition, we ``assign'' a pole to each of the prescribed zeros, in order to control the convergence of the appropriate counterpart to the sum in \eqref{eq:intuitive}. We note in passing that this ``pairing'' of zeros and poles would obviously be inadmissible if our task were to construct $\zeta_{\chi}(s)$ with a given signed multiset of zeros and poles. 


\section{Zeros and poles on the $1$-line}\label{Sec:PNT}

The line $\sigma=1$ plays a special role in our subject for the simple reason that it is the abscissa of absolute convergence for the Dirichlet series of $\zeta_{\chi}(s)$. As far as universality is concerned, a deep and dramatic conclusion about this line may be drawn from Theorem~A in conjunction with what was observed in \cite{HLS} about vertical limits in $\Real s >1$:
Pick any $f(s)$ in $H^*(\Omega)$; then there exists a sequence of vertical translates $\zeta(s+i \tau_n)$, with $\tau_n$ in $\mathbb{R}$,  such that
\begin{itemize}
\item $\zeta(s+i\tau_n)\to f(s)$ uniformly on every compact subset of $1/2<\Real s < 1$,
\item $\zeta(s+i\tau_n)\to \zeta_{\chi}(s)$ uniformly on every compact subset of $\Real s>1$ for some $\chi$ on $\mathbb{T}^\infty$.
\end{itemize} 
Hence the vertical line $\Real s=1$ is a ``brick wall'' between uniform convergence on compact subsets of respectively the strip $1/2<\Real s<1$ and the half-plane $\Real s>1$. In this assertion, we could of course replace $\zeta(s)$ by any Helson zeta function that is universal for $H^*(\Omega)$.

With this situation in mind, we now establish a ``prime number theorem'' for our zeta functions $\zeta_{\chi}(s)$, displaying a different peculiarity  of the    $1$-line.
\begin{theorem}\label{thm:PNT}
Suppose that $\zeta_{\chi}(s)$ is meromorphic on the line $\Real s=1$. Then only the following three situations may occur:
\begin{itemize}
\item[(i)] $\zeta_{\chi}(s)$ has neither a pole nor a zero on $\sigma=1$.
\item[(ii)] $\zeta_{\chi}(s)$ has a simple pole and neither a zero nor any other pole on $\sigma=1$.
\item[(iii)] $\zeta_{\chi}(s)$ has a simple zero and neither a pole nor any other zero on $\sigma=1$.
\end{itemize}
\end{theorem}
\begin{proof}
In $\sigma>1$, we may represent $\zeta_{\chi}(s)$ by its Euler product. It follows that we have
\[ \log \zeta_{\chi}(s)=\sum_{p} \chi(p) p^{-s} +O(1) \]
uniformly in $\sigma>1$. Since 
\[ \sum_{p} p^{-\sigma} = \log \frac{1}{\sigma-1} + O(1), \]
it is clear that a pole or a zero on $\sigma=1$ must be simple. 

Now suppose we have a simple pole at $s=1+it_0$. Then $\log \zeta(s)-\log \zeta_{\chi}(s+i t_0)$ is analytic at  $s=1$. Representing this function by its Dirichlet series, we see that
\begin{equation} \label{eq:sump} \sum_{p} (1-\chi(p)p^{-it_0}) p^{-\sigma}=O(1) \end{equation}
uniformly for $\sigma>1$ . Writing $\chi(p)p^{-it_0}=:e^{i\theta_p}$ with $-\pi < \theta_p\le \pi$, we see that
\[ \Real (1-\chi(p)p^{-it_0}) p^{-\sigma} =(1-\cos \theta_p) p^{-\sigma} \asymp \theta_p^2 p^{-\sigma}, \]
so that \eqref{eq:sump} implies
\begin{equation} \label{eq:psum}  \sum_p \theta_p^2 p^{-1}<\infty. \end{equation}

We may now write
\[ \log \zeta_{\chi}(s)=\log \zeta (s-it_0)+i \sum_{p} \sin \theta_p p^{-s+it_0} +O(1), \]
which holds uniformly for $\sigma>1$.  By the Cauchy--Schwarz inequality and \eqref{eq:psum}, 
\begin{equation} \label{eq:pp} \sum_{p} \left|\sin \theta_p p^{-s+it_0}\right|\ll \left(\sum_{p} p^{1-2\sigma}\right)^{1/2} 
\sim \left(\log \frac{1}{(\sigma-1)}\right)^{1/2}. \end{equation}
Since $\zeta(s-it_0)$ has only one simple pole and no zeros on $\sigma=1$, the bound in \eqref{eq:pp} implies
that $\zeta_{\chi}(s)$ has neither an additional pole nor a zero on the line $\sigma=1$. 

An obvious variation of this argument applies when $\zeta_{\chi}(s)$ has a simple zero instead of a simple pole at the point $1+it_0$. 
\end{proof}

Theorem~\ref{thm:PNT} will be used several times in what follows, and it will in particular allow us to establish a general assertion about zero-free regions in Subsection~\ref{sec:zerofree}.

\section{A condition for universality} 
\label{sec:univ}
In this section, we identify the key ingredients required to establish Voronin universality and also the equivalence between the Riemann hypothesis and strong recurrence (see Theorem~B).  We recall the central points of our approach, mentioned in the introduction: We focus on purely multiplicative conditions for Voronin universality, and of central importance are bounded mean squares of $\log \zeta_{\chi}(s)$ and what we will call the Bohr--Landau condition for the density of the zeros and poles of $\zeta_{\chi}(s)$ in $\mathbb{C}_{1/2}$. 

Before presenting our general theorem on universality, we note that the proof of Theorem~A may be applied without any change to establish a condition in terms of mean squares of the function $\zeta_{\chi}(s)$ itself. Here we introduce the notation
\[ \mathbb{C}_{\alpha}:=\big\{ s=\sigma+it: \ \sigma>\alpha \big\} \] 
and the terminology that $\zeta_{\chi}(s)$ is of finite order in $\overline{\mathbb{C}_{\alpha}}$ for $\alpha<1$ if $\zeta_{\chi}(s)$ has $\sigma(\chi)\le \alpha$, is analytic in $\alpha \le \Real s < 1$, and satisfies  $|\zeta_{\chi}(\sigma+it)|=O\big(|t|^{A}\big)$ for some $A\ge 0$, uniformly in $\sigma\ge \alpha$. 
Functions of finite order constitute a classical subject in the theory of Dirichlet series (see for example \cite[p. 298]{T1}), where one usually requires the function to be analytic in $\overline{\mathbb{C}_{\alpha}}$. In view of Theorem~\ref{thm:PNT}, we have found it convenient to allow our functions to have a simple pole on the $1$-line, so that 
$\zeta(s)$ itself can be viewed as a function of finite order in any half-plane $\overline{\mathbb{C}_{\alpha}}$ for $\alpha<1$.
 \begin{theorem}\label{thm:vor1}
Suppose that $\sigma(\chi)\le 1/2$ and that $\zeta_{\chi}(s)$ is of finite order in  $\overline{\mathbb{C}_{\alpha}}$ and satisfies
\begin{equation} \label{eq:l22} \sup_{T\ge 1} \frac{1}{2T} \int_{-T}^T \left|\zeta_\chi(\alpha+it)\right|^2 dt <\infty \end{equation}
whenever $1/2<\alpha<1$.
Then $\zeta_{\chi}(s)$ is universal for $H^*(\Omega)$.
\end{theorem}
We will not comment further on the direct proof of this result, because we will establish later that it is a consequence of the main theorem of this section. The idea for this new result is  essentially to  replace $\zeta_{\chi}(s)$  by $\log \zeta_{\ch}(s)$ in \eqref{eq:l22}. This will result in a much weaker growth condition on $\zeta_{\chi}(s)$, and it will allow us to treat zeros and poles on equal terms. We need however, as already mentioned in the introduction, to add a density condition on the zeros and poles that holds automatically on the assumptions of Theorem~\ref{thm:vor1}. To this end, we set
\[ N(\chi, \sigma, T):= N_{Z(\zeta_{\chi}(s))}(\sigma, T)=
\sum_{\rho=\beta+i\gamma\in Z\left(\zeta_{\chi}(s)\right): \beta>\sigma, |\gamma|\le T} |m_{Z(\zeta_{\chi}(s))}(\rho)|\]
for $\sigma\ge \sigma(\chi)$. In the special case when  $\sigma(\chi)\le 1/2$, we say that $\zeta_{\chi}(s)$ satisfies the Bohr--Landau condition if
\begin{equation} \label{eq:blcond} N(\chi, \sigma, T)=o(T) \end{equation}
for every $\sigma>1/2$.

We will use the natural convention for a Helson zeta function $\zeta_{\chi}(s)$ with $\sigma(\chi)\le \alpha$ that
$\log \zeta_{\chi}(s)$ is the function defined in the domain obtained from $\mathbb{C}_{\alpha}$ by removing all horizontal line segments between the line $\Real s =\alpha$ and the zeros and the poles, if any, of $\zeta_{\chi}(s)$, by analytic continuation from the half-plane $\mathbb{C}_{1}$ of the Dirichlet series
\[ \sum_{n=2}^{\infty} \frac{\Lambda(n)}{\log n} \chi(n) n^{-s}. \]
Here and in the sequel, $\Lambda(n)$ denotes the classical von Mangoldt function which takes the value $\log p$ if $n=p^k$ for some $k\ge 1$ and otherwise $\Lambda(n)=0$. We notice that $\log \zeta_{\chi}(s)$ fails to exist only on at most a discrete subset  of any vertical line in $\mathbb{C}_{\alpha}$, and hence we may compute mean squares along such lines. These mean squares will all be finite since $\log \zeta_{\chi}(s)$ have only logarithmic singularities.

Voronin universality deals primarily with approximation properties of finite Euler products, and hence we are particularly interested in the products 
\[ P_x \zeta_{\chi}(s):= \prod_{p\le x} \frac{1}{(1-\chi(p)p^{-s})}\]
for which
\[ \log P_x \zeta_{\chi}(s):=\sum_{n\ge 2: p|n \Rightarrow p\le x} \frac{\Lambda(n)}{\log n} \chi(n) n^{-s}= \sum_{p\le x} \sum_{j=1}^{\infty} j^{-1}\chi(p)^j p^{-js} .\]
This Dirichlet series converges absolutely for $\Real s>0$. 

Our condition for Voronin universality now reads as follows.

\begin{theorem}\label{thm:vor2}
Suppose that $\sigma(\chi)\le 1/2$ and that $\zeta_{\chi}(s)$ satisfies the Bohr--Landau condition. If, in addition, there exists a constant $C$, depending only on $\chi$, such that
\begin{equation} \label{eq:logint}  \limsup_{T\to \infty} \frac{1}{2T} \int_{-T}^{T} \left|\log \zeta_\chi(\sigma+it)-\log P_x \zeta_{\chi}(\sigma+it)\right|^2 dt \le
C \sum_{p>x} p^{-2\sigma} \end{equation}
for $x\ge 1$, uniformly for $\sigma\ge \sigma_0$ whenever $\sigma_0>1/2$, then $\zeta_{\chi}(s)$ is universal for $H^*(\Omega)$.
\end{theorem}

Theorem~\ref{thm:BLthm} of the next section shows that the condition of Theorem~\ref{thm:vor1} implies that of Theorem~\ref{thm:vor2}. Hence, in view of Theorem~\ref{thm:BLthm}, the following  consequence of Theorem~\ref{thm:vor2} yields an extension of Theorem~B.  

\begin{coro}
Suppose that $\zeta_{\chi}(s)$ is a Helson zeta function satisfying the conditions of Theorem~\ref{thm:vor2}.
Then $\zeta_{\chi}(s)$ is a strongly recurrent point for vertical translations in
$M(\Omega)$ if and only if $\zeta_{\chi}(s)$ is in $H^*(\Omega)$.
\end{coro}

\begin{proof}
The ``if part'' is immediate from Theorem~\ref{thm:vor2}. To see that the ``only if part'' also holds, we use Rouch\'{e}'s theorem as in the proof of 
\cite[Thm. 4.7]{Ba} to show that if $\zeta_{\chi}(s)$ is a strongly recurrent point for vertical translations in $M(\Omega)$ and has a zero or a pole in $\Omega$, then there exists a $\sigma>1/2$ and a positive constant $c$ such that $N(\chi,\sigma,T)\ge c T$ for large enough $T$. This is in conflict with the Bohr--Landau condition \eqref{eq:blcond}, hence $\zeta_{\chi}(s)$ must belong to $H^*(\Omega)$.
\end{proof}

To prove Theorem~\ref{thm:vor2}, we will follow \cite[Ch. 11]{BM}. We begin by stating the crucial approximation property of finite Euler products. 
\begin{lemma}\label{lem:uni}
Let $f(s)$ be a function in $H^*(\Omega)$, and let $K$ be a compact subset of $\Omega$. Given $\varepsilon$, $\theta>0$ and any $\chi$ in $\T^{\infty}$, there exist a set $A$ of positive numbers with positive density and a positive number $X$ such that
\begin{enumerate}
\item $\sup_{s\in K} \left|f(s)-P_X \zeta_{\chi}(s+i\tau) \right|<\varepsilon$ for every $\tau$ in $A$;
\item for every $x\ge X$,
\[  \underline{d}\left(\left\{\tau\in A: \  \sup_{s\in K} \left| P_x \zeta_{\chi}(s+i\tau)-P_X \zeta_{\chi}(s+i\tau)\right|<\varepsilon \right\}\right) > (1-\theta) d(A). \]
\end{enumerate}  
\end{lemma}
We would like to stress that this remarkable result, originating in Voronin's work \cite{V}, is valid for every $\chi$, without any assumption on the function $\zeta_{\chi}(s)$. The proof is word for word the same as that of \cite[Thm. 11.2]{BM}, which in turn relies on \cite{Ba}.

\begin{proof}[Proof of Theorem~\ref{thm:vor2}]
We are given a compact set $K$ in $\Omega$ and begin by picking a bounded domain $U$, $K\subset U$, 
whose closure is contained in $\Omega$. We set
\[ \| g\|^2_{A^2(U)}:=\int_U |g(s)|^2 dm_2(s), \]
where $g(s)$ is some measurable function defined on $U$ and $m_2$ is Lebesgue area measure on $\mathbb{C}$.  The Bergman space $A^2(U)$ consists of those analytic functions $g(s)$ on $U$ for which $\| g\|_{A^2(U)}<\infty$. It is a well known fact (see \cite[Lem. 4.8.6]{BG}) that there exists a constant $C(K,U)$ such that
\begin{equation} \label{eq:bergman}  \max_{s\in K} |g(s)| \le C(K,U) \| g\|_{A^2(U)} \end{equation}
for every $g(s)$ in $A^2(U)$. 

Now set
\[ D:=\left\{ \tau>0: \ \log \zeta_{\chi} (s+i\tau) \in A^2(U)  \right\}; \]
we notice that by the Bohr--Landau condition, $d(D)=1$. Using our assumption \eqref{eq:logint} and Fubini's theorem, we see that we can make 
\[ \limsup_{T\to\infty} \frac{1}{T} \int_{\tau\in D, \tau\le T} \int_{U}  \left|\log \zeta_{\chi}(s+i\tau)-\log P_x \zeta_{\ch}(s+i\tau)\right|^2 dm(s) d\tau \]
as small as we wish if we choose $x$ large enough. Hence, in view of \eqref{eq:bergman}, we have
\[ \limsup_{T\to\infty} \frac{1}{T} \int_{\tau\in D, \tau\le T} \max_{s\in K} \left|\log \zeta_{\chi}(s+i\tau)-\log P_x \zeta_{\ch}(s+i\tau)\right|^2 d\tau < \eta\]
for $x$ sufficiently large, given an arbitrary $\eta>0$. We infer from this, by Chebyshev's inequality, that
\[ m\left(\left\{\tau\in D, \tau\le T: \ \max_{s\in K} \left|\log \zeta_{\chi}(s+i\tau)-\log P_x \zeta_{\ch}(s+i\tau)\right| \ge \varepsilon \right\}\right) 
\le \left(2\eta/ \varepsilon^2\right) T\]
for $T$ large enough. Since $d(D)=1$, this entails that
\begin{equation} \label{eq:epseta} m\left( \left\{\tau\in D, \tau\le T: \ \max_{s\in K} \left|\log \zeta_{\chi}(s+i\tau)-\log P_x \zeta_{\ch}(s+i\tau)\right| < \varepsilon \right\}\right)
\ge \left(1- 3\eta/\varepsilon^2\right) T \end{equation}
for sufficiently large $T$. We now observe that if $z$ and $w$ are two arbitrary complex numbers, then
\[ |z-w|<\varepsilon \quad \text{implies} \quad \left|\exp\left(z-w\right)-1\right|<e^{\varepsilon}-1. \]
Therefore, since we can make $\eta$ as small as we wish by choosing $x$ sufficiently large, \eqref{eq:epseta} implies that, given an arbitrary compact subset $K$ of $\Omega$ and $\varepsilon, \delta>0$, there exists a positive number $Y$ such that
\begin{equation} \label{eq:finite} \underline{d}\left( \left\{\tau>0: \ \max_{s\in K}\left|1-\frac{\zeta_{\chi}(s+i\tau)}{P_x \zeta_{\chi}(s+i\tau)}\right|<\varepsilon \right\}\right)\ge 1-\delta \end{equation}
whenever $x\ge Y$.

Now let $f(s)$ be any function in $H^{*}(\Omega)$ and $K$ any compact subset of $\Omega$. We apply Lemma~\ref{lem:uni} with $\theta=1/3$ and an arbitray $\varepsilon>0$.
Accordingly, there exist a set $A$ of positive density and a positive number $X$ such that
the set
\[ A_x:=\left\{\tau\in A:  \ \max_{s\in K} \left|P_x \zeta_{\chi}(s+i\tau)-P_X \zeta_{\chi}(s+i\tau)\right|<\varepsilon \right\} \]
satisfies
\[ \underline{d}(A_x) \ge 2d(A)/3 \]
whenever $x\ge X$. Moreover, using also conclusion (1) of Lemma~\ref{lem:uni} and setting
\[ M:=\max_{s\in K} |f(s)|+2\varepsilon , \]
we have
\begin{equation} \label{eq:max} \max_{s\in K} \left|P_x \zeta_{\chi}(s+i\tau)\right|<M \end{equation}
for every $\tau$ in $A_x$ when $x\ge X$. Then setting 
\[ B_x:=\left\{\tau >0: \ \max_{s\in K}\left|1-\frac{\zeta_{\chi}(s+i\tau)}{P_x \zeta_{\chi}(s+i\tau)}\right|<\varepsilon \right\} ,\]
we may infer from the triangular inequality that 
\[ \max_{s\in K} \left|\zeta_{\chi}(s+i\tau)-f(s) \right| < 2\varepsilon+M \varepsilon  \]
when $\tau$ is in $A_x\cap B_x$ and $x \ge X$. Choosing $\delta=d(A)/3$ in \eqref{eq:finite}, we find that $\underline{d}(B_x)\ge 1-d(A)/3$ for $x\ge Y$, and we therefore have
\begin{align*}  \underline{d}\left(\left\{ \tau>0: \  \max_{s\in K} \left|\zeta_{\chi}(s+i\tau)-f(s) \right| < 2\varepsilon+M \varepsilon   \right\} \right) & \ge \underline{d}(A_x\cap B_x) \\ & \ge \underline{d}(A_x)+\underline{d}(B_x) - 1 \ge d (A)/3\end{align*} 
when both $x \ge X$ and $x\ge Y$. This concludes the proof, since $\varepsilon$ may be suitably adjusted. 
\end{proof}

We close this section by observing that we could have dropped the proviso that $\mathcal{Z}$ satisfy the Bohr--Landau condition in Theorem~\ref{thm:vor2}, if we required
$\mathcal{Z}$ to be a multiset instead of a signed multiset. Indeed, setting
\[ n_{\mathcal{Z}}(\sigma,T):=\sum_{\rho=\beta+i\gamma\in\mathcal{Z}, \beta>\sigma, T< |\gamma|\le T} m_{\mathcal{Z}}(\rho), \]
we may use a classical formula of Littlewood \cite[(9.9.1) p. 220]{T} to deduce that
\begin{equation} \label{eq:Little} \int_{R(\sigma,T)} \Big(\log \zeta_{\chi}(s)-\log P_x \zeta_{\chi}(s) \Big) ds = -2\pi i \int_{\sigma}^2 n_{\mathcal{Z}}(u,T) du, \end{equation}
where $R(\sigma,T)$ is the contour obtained by traversing the boundary of the rectangle 
\[ \big\{u+iv:\ \sigma\le u \le 2, |v|\le T \big\} \] 
in the counterclockwise direction. Assuming that \eqref{eq:logint} holds uniformly for $\sigma\ge \sigma_0$ whenever $\sigma_0>1/2$, we find that
\[ \int_{-T-1}^{T+1} \int_{\sigma_0}^2 \left|\log \zeta_\chi(\sigma+it)-\log P_x \zeta_{\chi}(\sigma+it)\right|^2 dt d\sigma \le
C T \sum_{p>x} (\log p)^{-1} p^{-2\sigma_0} \]
when $T$ is large enough. Hence there exists a $\xi$ in $[0,1]$ such that
\[  \int_{\sigma_0}^2 \left|\log \zeta_\chi(\sigma\pm i(T+\xi))-\log P_x \zeta_{\chi}(\sigma\pm i(T+\xi))\right|^2  d\sigma \le
C T \sum_{p>x} (\log p)^{-1} p^{-2\sigma}. \]
Returning to \eqref{eq:Little} and applying the Cauchy--Schwarz inequality along each side of the rectangle, we infer that
\begin{equation}\label{eq:Nbound} \Bigg| \int_{\sigma}^2 n_{\mathcal{Z}}(u,T) du \Bigg| \ll T \Bigg(\sum_{p>x} p^{-2\sigma_0} \Bigg)^{1/2}. \end{equation}
If $\mathcal{Z}$ is a multiset, then $n_{\mathcal{Z}}(u,T)=N_{\mathcal{Z}}(u,T)$, so that \eqref{eq:Nbound} entails that
\[ N_{\mathcal{Z}}(\sigma,T) \ll  T \Bigg(\sum_{p>x} p^{-2\sigma_0} \Bigg)^{1/2} \]
for every $\sigma>\sigma_0$. This holds for every fixed $x\ge 1$ and sufficiently large $T$, and hence it implies the Bohr--Landau condition. 

Because of possible cancellations in the sum defining $n_{\mathcal{Z}}(u,T)$, we may not conclude similarly from \eqref{eq:Nbound} in the general case. A variant of the constructions of Section~\ref{sec:main} and Section~\ref{sec:proofany} may in fact be used to show that \eqref{eq:logint} may hold even if the Bohr--Landau condition fails. Hence the condition that $\mathcal{Z}$ satisfy the Bohr--Landau is not obsolete in Theorem~\ref{thm:vor2}.
\section{Helson zeta functions of bounded type}\label{sec:finite}

\subsection{Functions of bounded type and the canonical factorization} We will now establish some useful facts about Helson zeta functions that are of bounded type in some half-plane. This discussion will in particular allow us to prove, in the next subsection, that Theorem~\ref{thm:vor1} is a consequence of Theorem~\ref{thm:vor2} . 

To begin with, we recall that a meromorphic function $h(s)$ in a domain $D$ of the complex plane is said to be a function of bounded type if we may write 
$h(s)=f(s)/g(s)$, with $f(s), g(s)$ bounded analytic functions in $D$. It is a classical fact that a function of bounded type in some half-plane $\mathbb{C}_{\alpha}$ admits the following canonical factorization (see for example \cite[p. 197]{Ne}). First, a signed multiset $\mathcal{Z}$ in $\mathbb{C}_{\alpha}$ will constitute the zeros and poles of some function $h(s)$ of bounded type in $\mathbb{C}_{\alpha}$ if and only it satisfies the Blaschke condition
\begin{equation} \label{eq:blaschke}  \sum_{\rho=\beta+i\gamma \in \mathcal{Z}} |m_{\mathcal{Z}}(\rho)| \frac{(\beta-\alpha)}{1+\gamma^2} < \infty .\end{equation}
Based on this fact, we introduce 
functions of the form
\[ b_{\rho;\alpha}(s):= \left(\frac{s-\rho}{s-2\alpha+\overline{\rho}}\right)\cdot \left(\frac{1-\alpha+\overline{\rho}}{1+\alpha-\rho} \right)\cdot \left|\frac{1+\alpha-\rho} {1-\alpha+\overline{\rho}}\right| \]
as the ``atoms'' of our representation of zeros and poles when $\rho\neq 1+\alpha$.  
The function $b_{\rho; \alpha}(s)$ is a conformal map of $\mathbb{C}_{\alpha}$ to the unit disc, sending $\rho$ to $0$, normalized to make the (generalized) Blaschke product
\[ B_{\mathcal{Z};\alpha}(s):=\left(\frac{s-1-\alpha}{s+1-\alpha}\right)^{m_{\mathcal{Z}}(1+\alpha)}\prod_{\rho \in \mathcal{Z}, \rho\neq 1+\alpha} \left[b_{\rho;\alpha}(s)\right]^{m_{\mathcal{Z}}(\rho)} \] 
absolutely convergent for every $s$ in $\mathbb{C}_{\alpha}$ when \eqref{eq:blaschke} holds. 

We use also the fact that $h(\sigma+it)$ tends to a finite boundary value, called $h(\alpha+it)$, when $\sigma\searrow \alpha$ for almost every point $t$ of 
the real line. In fact, this boundary function will satisfy 
\[ \int_{-\infty}^{\infty} \frac{|\log |h(\alpha+it)||}{1+t^2} dt < \infty, \]
which allows us to introduce the outer function
\begin{equation} \label{eq:G} U(s):=\exp\left(\frac{1}{\pi} \int_{-\infty}^{\infty} \left[ \frac{1}{s-\alpha-ix}-\frac{ix}{1+x^2}\right]\log|h(\alpha+i x)| dx\right).\end{equation}
In general, the canonocial factorization takes the form
\begin{equation} \label{eq:factor}  h(s)=B_{\mathcal{Z};\alpha}(s) S(s) e^{a(s-\alpha)+ib} U(s), \end{equation}
where $a$ and $b$ are real numbers and $S(s)$ is the ratio of two singular inner functions, represented by a singular measure on the line $\Re s =\alpha$.

We will only be interested in the case when $h(s)$ is a Helson zeta function $\zeta_{\chi}(s)$ that extends meromorphically to the closed half-plane 
$\overline{\mathbb{C}_{\alpha}}$. Then the factor $S(s)$ will not be present in the canonical factorization of $h(s)$. Also, since $\log |\zeta_{\chi}(\sigma)| \to 0$ when $\sigma\to \infty$ and we may show that the remaining part of the product will give a contribution of size $o(\sigma)$ to $\log |\zeta_{\chi}(\sigma)|$, it is clear that $a=0$. The unimodular factor $e^{ib}$ may now be absorbed in $B_{\mathcal{Z};\alpha}(s)$ so that \eqref{eq:factor} reduces to
\[ \zeta_{\chi}(s)=B_{\mathcal{Z};\alpha}(s)  U(s) \]
when $\zeta_{\chi}(s)$ is of bounded type in $\mathbb{C}_{\alpha}$ and meromorphic in $\overline{\mathbb{C}_{\alpha}}$. 

It is immediate that if $\zeta_{\chi}(s)$ is of finite order in $\mathbb{C}_{\alpha}$, then it is also of bounded type. Indeed, if $\zeta_{\chi}(s)$ has no pole on the $1$-line, then $\zeta_{\chi}(s)$ is the ratio of the two bounded analytic functions $\zeta_{\chi}(s)/(s+\alpha+1)^A$ and $1/(s+\alpha+1)^A$ for a suitable $A\ge 0$. If $\zeta_\chi (s)$ has a simple pole at $1+it_0$, we just multiply each of these function by $s-1-it_0$ and get the same result. 

It requires a little more to see that the functions constructed in Theorem~\ref{thm:main} and Theorem~\ref{thm:main3}, via the conditions of Theorem~\ref{thm:vor2}, will also be of bounded type. This fact is of some general interest, but since we will not use it in the sequel, we only sketch the argument. We observe to begin with that plainly, if $\mathcal{Z}$ satisfies condition (c) of Theorem~\ref{thm:main}, then the Blaschke condition \eqref{eq:blaschke} holds for every $\alpha>1/2$. Next, by the Cauchy--Schwarz inequality,
\[ \int_{\sigma-i\infty}^{\sigma+i\infty} \frac{|\log h(\sigma+it)|}{1+t^2} dt \le 
\sqrt{\pi} \left(\int_{\sigma-i\infty}^{\sigma+i\infty} \frac{|\log h(\sigma+it)|^2}{1+t^2} dt\right)^{1/2}, \]
and furthermore,
\begin{align*} \int_{\sigma-i\infty}^{\sigma+i\infty} \frac{|\log h(\sigma+it)|^2}{1+t^2} dt & 
\ll \sum_{k=0}^\infty 2^{-2k} \int_{\sigma-i2^k}^{\sigma+i 2^k} |\log h(\sigma+it)|^2 dt \\
& \le 2 \sum_{k=0}^\infty 2^{-k} \sup_{T\ge 1}\frac{1}{2T} \int_{\sigma-iT}^{\sigma+i T} |\log h(\sigma+it)|^2 dt. \end{align*}
These two bounds show that \eqref{eq:logint} with $x=1$ implies
\[ 
\sup_{\sigma\ge\alpha} \int_{-\infty}^{\infty} \frac{\left|\log \zeta_{\chi}(\sigma+it)\right|}{1+t^2} dt < \infty. \]
We may use the Blaschke condition to show that, similarly,
\[ \sup_{\sigma\ge\alpha} \int_{-\infty}^{\infty} \frac{\left|\log B_{\mathcal{Z};\alpha}(\sigma+it)\right|}{1+t^2} dt < \infty, \]
It then follows that the function
\[ h_R(s):=\left[\frac{\zeta_{\chi}(s)}{B_{\mathcal{Z};\alpha}(s)}\right]^{\left(\frac{R+\alpha}{R+s}\right)^2} \]
is a bounded zero-free function in the half-plane $\mathbb{C}_{\alpha+\varepsilon}$ for every $R>-\alpha$ and $\varepsilon>0$, and hence
\[ h_R(s)=
\exp\left(\frac{1}{\pi} \int_{-\infty}^{\infty} \left[ \frac{1}{s-\alpha-\varepsilon-ix}-\frac{ix}{1+x^2}\right]\log\left|h_R(\alpha+\varepsilon+ i x)\right| dx\right) \] 
when $0<\varepsilon<\sigma-\alpha$. We now let $R \to\infty$ and use Lebesgue's dominated convergence theorem to deduce that
\[ \frac{\zeta_{\chi}(s)}{B_{\mathcal{Z};\alpha}(s)}=\exp\left(\frac{1}{\pi} \int_{-\infty}^{\infty} \left[ \frac{1}{s-\alpha-\varepsilon-ix}-\frac{ix}{1+x^2}\right]\log\left|
\frac{\zeta_{\chi}(\alpha+\varepsilon+ i x)}{B_{\mathcal{Z};\alpha}(\alpha+\varepsilon+ix)}\right| dx\right). \]
Since this holds for every $\varepsilon>0$, we may use a normal family argument to infer that $\zeta_{\chi}(s)/B_{\mathcal{Z};\alpha}(s)$ is the ratio of two bounded analytic functions in $\mathbb{C}_{\alpha}$. This is exactly what we need to conclude that $\zeta_{\chi}(s)$ is of bounded type in 
$\mathbb{C}_{\alpha}$.
 
We turn to some bounds that in particular illustrate that much stronger conclusions may be drawn if we assume $\zeta_{\chi}(s)$ to be of finite order rather than just of bounded type. 
 
\begin{lemma} \label{lem:innout}
Suppose that $\sigma(\chi)\le \alpha<1$ and that $\zeta_{\chi}(s)$ is meromorphic in $\overline{\mathbb{C}_{\alpha}}$ and of bounded type in $\mathbb{C}_{\alpha}$. Then 
the outer function $U(s)$ of the canonical factorization of $\zeta_{\chi}(s)$ satisfies
 \begin{equation} \label{eq:logt} \frac{1}{\pi} \int_{-\infty}^{\infty} \frac{|\log |U(\alpha+ix)||}{(x-t)^2+1} dx =
 \begin{cases} O( \log |t|), & \text{$\zeta_{\chi}(s)$ is of finite order in $\overline{\mathbb{C}_{\alpha}}$} \\
                                     o(t^2), & \text{otherwise} \end{cases} \end{equation} 
for $|t|\ge 2$.
\end{lemma} 
 \begin{proof}
From \eqref{eq:G} we find that
\begin{equation} \label{eq:Poiss}\log |U(\sigma+it)| = \frac{1}{\pi} \int_{-\infty}^{\infty} \frac{(\sigma-\alpha)\log|U(\alpha+ix)|}{(x-t)^2+(\sigma-\alpha)^2} dx. \end{equation}
We begin with the case when $\zeta_{\chi}(s)$ is of finite order. Then $\log |B_{\mathcal{Z};\alpha}(2+it)|\le C$ for some constant $C$ because $B_{\mathcal{Z};\alpha}(s)$ has at most one pole, which is simple and located on the $1$-line. Since we also have $\log |\zeta_{\chi}(2+it)|=O(1)$, it follows that
\[ |U(2+it)| \ge C \]
for some positive constant $C$.
In view of \eqref{eq:Poiss} and the bound $\log |U(\alpha+ix)|\le c\log (|x|+2)$, we conclude that
\eqref{eq:logt} holds. 

When $\zeta_{\chi}(s)$ is of bounded type, we just observe that
\[  \int_{-\infty}^{\infty} \frac{|\log |U(\alpha+ix)||}{(x-t)^2+1} dx \ll t^2 \int_{|x|\ge |t|/2} \frac{|\log |U(\alpha+ix)||}{1+x^2} dx +O(1)=o(t^2). \]\end{proof}

Lemma~\ref{lem:innout} yields the following ``explicit formula''.
\begin{lemma}\label{lem:prime}
Let the assumptions be as in Lemma~\ref{lem:innout}. Then
\[ \frac{\zeta'_{\chi}(s)}{\zeta_{\chi}(s)}=\sum_{\rho} m_{\mathcal{Z}}(\rho) \left(\frac{1}{s-\rho}-\frac{1}{s-2\alpha+\overline{\rho}}\right)
+\frac{U'(s)}{U(s)}, \]
where 
\[ \frac{U'(s)}{U(s)}=\begin{cases}  O(\log |t|), & \text{$\zeta_{\chi}(s)$ of finite order} \\
                                                 o\left(t^2\right), &  \text{otherwise}\end{cases}\]
when $|t|\to \infty$ and $\Real s\ge \sigma_0$ for every fixed $\sigma_0>\alpha$.
\end{lemma}

\begin{proof}
We see from \eqref{eq:G} that
\[ \frac{U'(s)}{U(s)}=-\frac{1}{\pi} \int_{-\infty}^{\infty} \frac{\log |U(\alpha+ix)|}{(s-\alpha-ix)^2} dx. \]
Hence by \eqref{eq:logt}, 
\[  \frac{U'(s)}{U(s)}= \begin{cases}  O(\log |t|), & \text{$\zeta_{\chi}(s)$ of finite order} \\
                                                 o\left(t^2\right), &  \text{$\zeta_{\chi}(s)$ is of bounded type}\end{cases} \]
when $\sigma\ge \sigma_0$ for some fixed $\sigma_0 >\alpha$. Since the sum is just the logarithmic derivative of $B_{\mathcal{Z};\alpha}(s)$, the result follows from the canonical factorization.
\end{proof}
The bound for $U'(s)/U(s)$ may seem rather inordinate in the case when $\zeta_{\chi}(s)$ is assumed only to be of bounded type. We will nevertheless be able to make use of this estimate in Subsection~\ref{sec:zerofree} to draw some nontrivial conclusions about zero-free regions.

\subsection{Proof that Theorem~\ref{thm:vor2} implies Theorem~\ref{thm:vor1}}

We begin by establishing the following simple fact.
\begin{lemma}\label{lem:zeros}
Suppose that $\sigma(\chi)\le \alpha<1$ and that $\zeta_{\chi}(s)$ is of finite order in $\overline{\mathbb{C}_{\alpha}}$. Then for every $\sigma>\alpha$,
\[ N(\chi,\sigma,T+1)-N(\chi,\sigma, T)=O(\log T). \]
\end{lemma}
\begin{proof}
We use again Theorem~\ref{thm:PNT} according to which $\zeta_{\chi}(s)$ has at most one pole on the $1$-line. This means that, away from this possible pole, we may apply Jensen's formula in discs of radius $1-\sigma/2-\alpha/2$ centered on the $1$-line to conclude in a similar way as in \cite[Thm. 13.5]{T}. 
\end{proof}

The preceding lemmas enable us to show that Theorem~\ref{thm:vor2} implies Theorem~\ref{thm:vor1}. 

\vbox{\begin{theorem}\label{thm:BLthm}
Suppose that $\sigma(\chi)\le 1/2$ and that $\zeta_{\chi}(s)$ is of finite order in  $\overline{\mathbb{C}_{\alpha}}$ and satisfies
\begin{equation} \label{eq:l2} \sup_{T\ge 1} \frac{1}{2T} \int_{-T}^T \left|\zeta_\chi(\alpha+it)\right|^2 dt <\infty \end{equation}
whenever $1/2<\alpha<1$.
Then $\zeta_{\chi}(s)$ satisfies the Bohr--Landau condition and
\begin{equation} \label{eq:secl} \lim_{T\to \infty}\frac{1}{2T} \int_{-T}^{T} \left|\log \zeta_\chi(\sigma+it)-P_x \log \zeta_{\chi}(\sigma+it)\right|^2 dt = \ \sum_{p>x} \sum_{j=1}^{\infty} j^2  p^{-2j\sigma}, \end{equation} 
uniformly whenever $\sigma\ge \alpha$ and $\alpha>1/2$.
\end{theorem}}
\begin{proof}
We rely on familiar arguments and will therefore only sketch the proof. We may assume without loss of generality that a possible simple pole of $\zeta_{\chi}(s)$ is not located in the upper half-plane. Our first aim is to show that we will have
\begin{equation} \label{eq:BLest} N(\chi,\sigma,T)=O(T^{1-\varepsilon}) \end{equation}
for $\sigma>1/2$ and $\varepsilon=\varepsilon(\sigma)>0$. This follows by a classical argument of Ingham \cite{In} which may be used in the following way. Set
\[ M_X(s):=\sum_{n<X} \mu(n) \chi(n) n^{-s}, \]
where $\mu(n)$ is the M\"{o}bius function, and set also $f_X(s):=\zeta_{\chi}(s)M_X(s)-1$ and $Q_X(s):=1-f_X(s)^2$, so that the zeros of $\zeta_{\chi}(s)$ are also zeros of $Q_X(s)$. We notice that 
\[ \log |Q_X(s)|\le \log (1+|f_X(s)|^2)\le   |f_X(s)| \]
and estimate then
\[ \int_1^T \left|f_X\left(1/2+\eta+it\right)\right|dt \quad \text{and} \quad \int_1^T \left|f_X\left(1+1/\log T+it\right)\right| dt  \]
in the usual way for a suitable $\eta$, $0<\eta<\sigma-1$, using the Cauchy--Schwarz inequality in either case. Here it is crucial that the mean squares of $M_X(s)$ are uniformly bounded as long as we stay to the right of the $1/2$-line, so that, by our assumption \eqref{eq:secl}, the first of these integrals is $O(T)$, uniformly in $X$; the second integral is $O\big(T^{1/2}(T/X+1)(\log T)^2\big)$ by the bound (17) of \cite{ In}.
Now appealing to the same lemma of Littlewood that was used at the end of Section~\ref{sec:univ} (see also \cite[Sec. 9.9]{T}) and a convexity argument for the moments in a strip \cite{HIP}, we may conclude in a similar way\footnote{The explicit value for $\varepsilon=\varepsilon(\sigma)$, obtained by this argument, is irrelevant in our analysis, and we have made no attempt to optimize our proof sketch.} as in \cite{In}; see also \cite[pp. 230--235]{T}. 

We set $\mathcal{Z}:=Z(\zeta_{\chi}(s))$ and notice that Lemma~\ref{lem:zeros} and Lemma~\ref{lem:prime} give us a replica of a familiar formula in the theory of the Riemann zeta function, namely
\[ \frac{\zeta_{\chi}'(s)}{\zeta_{\chi}(s)} = \sum_{\rho: |\gamma-t|\le 1, \beta\ge \alpha} \frac{m_{\mathcal{Z}}(\rho)}{s-\rho}+ O\left(\log |t|\right) \]
uniformly for $\sigma\ge \alpha$. In fact, since $\log P_x \zeta_\chi (s) $ has an absolutely convergent Dirichlet series in $\sigma>0$, we get more generally
\[ \frac{\zeta_{\chi}'(s)}{\zeta_{\chi}(s)}-\frac{(P_x\zeta_{\chi})'(s)}{P_x\zeta_{\chi}(s)} = \sum_{\rho: |\gamma-t|\le 1, \beta\ge \alpha} \frac{m_{\mathcal{Z}}(\rho)}{s-\rho}+ O\left(\log |t|\right) \]
in the same half-plane $\sigma\ge \alpha$. Using this formula and \eqref{eq:BLest}, we may employ word for word the proof of \cite[Thm. 2]{BI} to conclude that \eqref{eq:secl} holds.
\end{proof}

\subsection{Zero-free regions}\label{sec:zerofree}
Our aim is now to investigate under how general conditions we may establish the classical zero-free region of de la Vall\'{e}e Poussin. It turns out that if there is a zero or pole on the $1$-line, then we can do with the rather weak condition that $\zeta_{\chi}(s)$ be of bounded type in some half-plane, along with a density condition for the zeros that in particular holds for the zeros of $\zeta(s)$. 

We begin by stating the following general theorem.

\vbox{\begin{theorem} \label{th:zero-free}
Suppose that $\sigma(\chi)\le \alpha<1$ and that $\zeta_{\chi}(s)$ is analytic in $\alpha < \Real s < 1$ and of bounded type in $\mathbb{C}_{\alpha}$. Assume, moreover, that the following conditions hold.
\begin{itemize}
\item[(a)] There exists a constant $A$ such that 
\[ 
\left| \frac{B_{\mathcal{Z};\alpha}'(1+1/\log |t|+it)}{B_{\mathcal{Z};\alpha}(1+1/\log |t|+it)} \right| \ll A\log |t| 
\]
when $|t|$ is large enough.
\item[(b)] There exists a positive constant $C$ such that for every positive $\delta$,
\begin{equation} \label{eq:cond} -\Real \frac{\zeta'_{\chi^2}(\sigma+it)}{\zeta_{\chi^2}(\sigma+i t)} \le C \log |t| \ \text{ for } \ \sigma>1+\delta/\log |t| \ \text{ and} \ |t| \text{ large enough.} \end{equation}
\end{itemize}
Then there exists a positive constant $c$ such that $\zeta_{\chi}(s)$ has no zero in the region 
\[ 1-\frac{c}{\log(|\Imag s|+2)} < \Real s <1. \] 
\end{theorem} }

Here it is important that $C$ does not depend on $\delta$, since otherwise our assumption \eqref{eq:cond} would be obsolete. It is fairly easy to establish that  \eqref{eq:cond} does hold whenever $\zeta_{\chi}(s)$ has a pole or a zero on the line $\sigma=1$. This leads to the following corollary.
\begin{coro} \label{cor:zero-free}
Suppose that $\sigma(\chi)\le \alpha<1$ and that $\zeta_{\chi}(s)$ has a zero or a pole at $s=1$. Assume also that $\zeta_{\chi}(s)$ is analytic in 
$\alpha < \Real s < 1$, of bounded type in $\mathbb{C}_{\alpha}$, and satisfies condition (a) of Theorem~\ref{th:zero-free}.
Then there exists a $c>0$ such that $\zeta_{\chi}(s)$ has no zero in the region
\[ 1-\frac{c}{\log(|\Imag s|+2)} < \Real s <1. \]  
\end{coro}

By Lemma~\ref{lem:prime}, it is clear that \eqref{eq:cond} is satisfied whenever $\zeta_{\chi^2}(s)$ also extends meromorphically to an analytic function of finite order in a strip to the left of the line $\sigma=1$. 
\begin{proof}[Proof of Theorem~\ref{th:zero-free}] 
We recall from Lemma~\ref{lem:prime} that
\begin{equation} \label{eq:Zdec} \frac{\zeta_{\chi}'(s)}{\zeta_{\chi}(s)} =\frac{B_{\mathcal{Z};\alpha}'(s)}{B_{\mathcal{Z};\alpha}(s)} + \frac{U'(s)}{U(s)} ,\end{equation}
where 
\[ \frac{U'(s)}{U(s)}=o(|t|) \]
uniformly in $\sigma\ge \alpha$. Our assumption on $B_{\mathcal{Z};\alpha}(s)$ and the trivial bound
\[ \left|\frac{\zeta_{\chi}'(1+1/\log |t|+it)}{\zeta_{\chi}(1+1/\log |t|+it)} \right| \le \log |t| +O(1) \]
also gives 
\[ \left| \frac{U'(1+1/\log|t|+it)}{U(1+1/\log |t| )} \right| \le D \log |t| \]
for some constant $D$. Now  applying the Hadamard three lines theorem to the function
\[ \frac{U'(s)}{U(s)(s+1-\alpha)} \]
for the three lines $\sigma=\alpha$,  $\sigma=1+1/\log |t|$, and $\sigma=1+\delta/\log |t|$ with $\delta<1$, we get
\begin{equation} \label{eq:Gest} \left| \frac{U'(1+\delta/\log|t|+it)}{U(1+\delta/\log |t| )} \right| 
\ll D' |t|^{\frac{1-\delta}{(\alpha-1)\log |t| -1}} \log |t| \le D'' \log |t|, \end{equation}
uniformly for $0\le \delta\le 1$ and $|t|$ sufficiently large.

Following Mertens, using the familiar inequality $3+4\cos \theta +\cos 2\theta \ge 0$, we find next that
\[ - 3 \frac{\zeta'(\sigma)}{\zeta(\sigma)} -4 \Real \frac{\zeta_{\chi}'(\sigma+it)}{\zeta_{\chi}(\sigma+it)}
-\Real \frac{\zeta_{\chi^2}'(\sigma+i2t)}{\zeta_{\chi^2}(\sigma+i2t)}\ge 0 \]
whenever $\sigma>1$ and $t$ is an arbitrary real number. We choose $\sigma=1+\delta/\log |t|$ for some $\delta$ to be chosen later and use \eqref{eq:cond} to infer that
\begin{equation} \label{eq:M}   4 \Real \frac{\zeta_{\chi}'(\sigma+it)}{\zeta_{\chi}(\sigma+it)} \le \left(\frac{3}{\delta}+B\right)\log |t| +O(1). \end{equation}

Now suppose there is a zero of $\zeta_{\chi}(s)$ at $\rho=\beta+i t$ with $\beta\ge \alpha+(1-\alpha)/2$. Then by \eqref{eq:Zdec} and \eqref{eq:Gest},
\[ \Real \frac{\zeta_{\chi}'(\sigma+it)}{\zeta_{\chi}(\sigma+it)}\ge \frac{1}{\sigma-\beta} - D\log |t|= \frac{1}{1-\beta+\delta/\log |t|} - D \log |t| \]
for some constant $D$.
Inserting this into \eqref{eq:M}, we find that
\[ \frac{4 \log |t|}{(1-\beta)\log|t|+\delta}\le \left(\frac{3}{\delta}+C+4D\right)\log |t| +O(1). \]
We now conclude in the usual way by choosing $\delta$ small enough, say $2\delta \le 1/(C+4D)$, so that $1-\beta\ge c/\log |t|$ for a small constant $c$ depending on $\delta$. 
\end{proof}
Curiously, our deduction of \eqref{eq:Gest} shows that condition (b) holds if $\zeta_{\chi^2}(s)$ meets the same conditions as those imposed on $\zeta_{\chi}(s)$ in Theorem~\ref{th:zero-free}.

\begin{proof}[Proof of Corollary~\ref{cor:zero-free}]
We need to check that \eqref{eq:cond} is satisfied. The crucial point will be to show that
\begin{equation} \label{eq:key}
\sum_{p\le x} |\theta_p| (\log p) p^{-1}=o(\log x), \end{equation} 
where have set $\chi(p)=:e^{i\theta_p}$, $-\pi<\theta\le \pi$. 
To this end, let $\varepsilon$ be an arbitrary
positive number, and set $P(k,\varepsilon):=\left\{p: \ 2^{2^k}<p \le 2^{2^{k+1}}, \ |\theta_p|\ge \varepsilon\right\}$. Then, by \eqref{eq:psum}, 
\[ \sum_{p\in P(k,\varepsilon)} \frac{|\theta_p|}{p} \ge \varepsilon \]
for at most a finite number of positive integers $k$. Hence  for all but finitely many $k$, we have
\[ \sum_{2^{2^k}<p\le 2^{2^{k+1}}} |\theta_p| (\log p) p^{-1} \le \varepsilon \cdot 2^{k+1}\left(1+\sum_{2^{2^k}<p\le 2^{2^{k+1}}}p^{-1}\right)\ll 
\varepsilon \cdot 2^{k+1},  \]
where we in the last step used Mertens's theorem $\sum_{p\le x} p^{-1}=\log\log x+O(1)$.
Summing over all $k$ such that $2^{2^{k+1}}\le x$, we arrive at \eqref{eq:key} since $\varepsilon$ can be chosen arbitrarily small.

It follows from the proof of Theorem~\ref{thm:PNT} that 
\[ \zeta_{\chi^2}(s)=\zeta(s) E(s), \]
where
\[  \frac{E'(s)}{E(s)}=\sum_p (\log p) c_p p^{-s} + O(1), \quad \sigma> 1, \]
and $|c_p|\ll |\theta_p|$. Now \eqref{eq:cond} follows from \eqref{eq:key} and the classical bound
\[ - \frac{\zeta'(s)}{\zeta(s)}\ll \log |t|, \]
which holds uniformly when $\sigma\ge 1$ and $|t|\ge 1$.  
\end{proof}

\section{Further preliminaries for the proofs of the main theorems}\label{sec:prelim}

It is convenient to collect in this section some additional auxiliary results that will be used in the proof of our main theorems. The first result was already mentioned in the introduction, but we state it here in a precise form for future reference.

\begin{lemma}\label{lem:gap}
There exists a positive constant $c$ such that  
\begin{equation} \label{eq:bhp} \pi(x+x^{21/40})-\pi(x) \ge c  x^{21/40}/\log x \end{equation}
for all sufficiently large $x$. On the Riemann hypothesis, 
\begin{equation} \label{eq:dudek} \pi\Big(x+c x^{1/2}\log x\Big)-\pi(x) \ge x^{1/2} \end{equation}
whenever $c>3$ and $x$ is sufficiently large.
\end{lemma}

\begin{proof} The unconditional part \eqref{eq:bhp} was established by Baker, Harman, and Pintz \cite[p. 562]{BHP}. The conditional part goes back to Cram\'{e}r \cite{Cr}; 
the present explicit bound \eqref{eq:dudek} was established by Dudek~\cite{AD}. \end{proof}

The next result, while almost evident, will be of basic importance for all the constructions to be made in the subsequent three sections.
\begin{lemma}\label{lem:exist}
Let $\Delta$ be a simply connected domain in $\mathbb{C}$ and $\mathcal{Z}$ a locally finite signed multiset in $\Delta$. If $R(s)$ is meromorphic in $\Delta$ and analytic in $\Delta\setminus \mathcal{Z}$ with simple poles at each of the points $\rho$ of $\mathcal{Z}$ of residue $m_{\mathcal{Z}}(\rho)$, then there exists a meromorphic function $F(s)$ in $\Delta$ with $Z(F(s))=\mathcal{Z}$ and such that $F'(s)/F(s)=R(s)$. The function $F(s)$ is unique up to multiplication by a nonzero constant. 
\end{lemma} 
\begin{proof}
Let $V(s)$ be an arbitrary meromorphic function in $\Delta$ with $Z(V(s))=\mathcal{Z}$. Such a function exists because  $\mathcal{Z}$ is assumed to be locally finite in $\Delta$. Then the function $V'(s)/V(s)-R(s)$ will be analytic in $\Delta$ and hence, since $\Delta$ is simply connected, the derivative of some analytic function $H(s)$ in $\Delta$. The function $F(s):=V(s)\exp(-H(s))$ will then have the required property. The uniqueness of $F(s)$ up to a multiplicative constant is obvious since the ratio of two functions with the same logarithmic derivative must be a constant. 
\end{proof}
The preceding lemma will be used in the following way. We will pick a sequence of primes $\mathcal{P}$ and select correspondingly unimodular numbers $\chi(p)$ for $p$ in $\mathcal{P}$. This defines $\chi(n)$ for every $n$ in $\mathbb{N}(\mathcal{P})$ which is the set of positive integers whose prime divisors are in $\mathcal{P}$. We then set 
\[ \zeta_{\chi; \mathcal{P}}(s):=\prod_{p\in \mathcal{P}} \frac{1}{(1-\chi(p) p^{-s})} \quad \text{and} \quad 
\log \zeta_{\chi;\mathcal{P}}(s):=\sum_{n\in \mathbb{N}(\mathcal{P})}\frac{ \chi(n)\Lambda(n)}{\log n} n^{-s} \]
and make the construction such that the logarithimic derivative
\begin{equation} \label{eq:logder} \frac{\zeta'_{\chi;\mathcal{P}}(s)}{\zeta_{\chi;\mathcal{P}}(s)}=-\sum_{n\in \mathbb{N}(\mathcal{P})} \chi(n)\Lambda(n) n^{-s} \end{equation}
extends meromorphically to $\mathbb{C}_{\sigma_0} $ for some $\sigma_0<1$,
with simple poles $\rho$ in $\mathcal{Z}$ and prescribed residues $m_{\mathcal{Z}}(\rho)$. Since
the equality in \eqref{eq:logder} holds for $\Real s>1$, Lemma~\ref{lem:exist} shows that also $\zeta_{\chi;\mathcal{P}}(s)$ has an analytic continuation to $\mathbb{C}_{\sigma_0}$ and
we have then $Z(\zeta_{\chi;\mathcal{P}}(s))\cap \mathbb{C}_{\sigma_0}=\mathcal{Z}$. 

We will also use the probabilistic model mentioned in the introduction to construct $\zeta_{\chi;\mathcal{P}}(s)$ that can be approximated by $P_x \zeta_{\chi;\mathcal{P}}(s)$ in the sense of \eqref{eq:logint} of Theorem~\ref{thm:vor2}.

\vbox{\begin{lemma}\label{lem:almost}
Let $\mathcal{P}$ be an arbitrary sequence of primes and $\chi(p)$, $p$ in $\mathcal{P}$, correspondingly constitute a sequence of independent Steinhaus variables such that
\[ \sum_{p\in \mathcal{P}} p^{-2\sigma} < \infty \] for every
$\sigma>\sigma_0$. Then almost surely
\begin{itemize}
\item[(i)] the Dirichlet series defining $\log \zeta_{\chi;\mathcal{P}}(s)$ 
converges for every $s$ in $\mathbb{C}_{\sigma_0}$,
\item[(ii)] \[ \lim_{T\to \infty} \frac{1}{2T} \int_{-T}^T \left|\log \zeta_{\chi;\mathcal{P}}(\sigma+it)-\log P_x \zeta_{\chi;\mathcal{P}}(\sigma+it) \right|^2 dt = \sum_{p\in \mathcal{P}, p>x} \sum_{j=1}^\infty j^{-2}p^{-2j\sigma} \]
uniformly for $\sigma\ge \alpha$ for every $\alpha>\sigma_0 $.
\end{itemize}
\end{lemma}}

\begin{proof} The lemma is an immediate consequence of \cite[Cor. 4.7]{HLS}. \end{proof}
 
We come finally to two estimates that will be used to check the mean square condition of Theorem~\ref{thm:vor2}. The first of these is a somewhat specialized version of a well-known inequality of Montgomery and Vaughan.

\vbox{\begin{lemma}\label{lem:secmom}
Let $\mathcal{N}$ be a sequence of positive integers and $\delta: [1,\infty)\mapsto [1,\infty)$  an associated nondecreasing function
such that 
\begin{equation} \label{eq:sepcond} \min_{n'\neq n, n'\in \mathcal{N}} |n-n'|\ge \delta(n) \end{equation}
for every $n$ in $\mathcal{N}$. If $a_n$, $n$ in $\mathcal{N}$, are arbitrary unimodular numbers, then
\[ \int_{-T}^{T}\left|\sum_{n\in \mathcal{N}, x<n\le y}a_n  n^{-\sigma-it}\right|^2 dt =2T \sum_{n\in \mathcal{N}, x<n\le y} n^{-2\sigma} 
+O\left(\int_1^{y+\delta(y)} [\delta(z)]^{-2} z^{1-2\sigma}dz\right), \]
uniformly for $\sigma\ge 1/2$, where the implicit constant on the right-hand side depends only on $\delta(z)$.
\end{lemma}}

\begin{proof}
If $1/2 \le n'/n \le 2$ and $n\neq n'$ for $n,n'$ in $\mathcal{N}$, then
\[ \left|\log n' - \log n\right|\ge \left(\log 2\right) \frac{|n-n'|}{n} \ge  \left( \log 2\right) \frac{\delta(n)}{n} =c\frac{\delta(n)}{n}.\]
The resulting inequality for $|\log n' - \log n |$, possibly after a slight adjustment of the constant $c$, then holds in general, without the precaution that $1/2 \le n'/n \le 2$. Hence, by an inequality of Montgomery and Vaughan \cite[Cor. 2]{MV},
\begin{equation}\label{eq:insert} \int_{-T}^{T}\left|\sum_{n\in \mathcal{N}, x<n\le y}a_n  n^{-\sigma-it}\right|^2 dt =2T \sum_{n\in \mathcal{N}, x<n\le y} n^{-2\sigma}
+ \Theta \sum_{n\in \mathcal{N}, x<n\le y} [\delta(n)]^{-1} n^{1-2\sigma},\end{equation}
where $  \left|\Theta\right|\le 6\pi/c$.
By the separation condition \eqref{eq:sepcond},
\[ \sum_{n\in \mathcal{N}, x<n\le y} [\delta(n)]^{-1} n^{1-2\sigma}\le \int_{1}^{y+\delta(y)}  [\delta(z)]^{-2} z^{1-2\sigma} dz,\]
where we also used that $z\mapsto [\delta(z)]^{-2} z^{1-2\sigma}$ is a decreasing function. 
Inserting this into \eqref{eq:insert}, we get the desired bound.
\end{proof}
The last lemma of this section is a simple integral variant of the preceding estimate.

\vbox{\begin{lemma}\label{lem:Isec}
Suppose that $|f(y)|\le \varphi(y)$ and $x\ge 1$, where $\varphi:[x,\infty)\to [1,\infty)$ is a nondecreasing function satisfying the doubling condition $\varphi(2y)\le C \varphi(y)$ and  the growth condition
\[ \int_{x}^{\infty} \varphi(y) y^{-\sigma_0-1} dy <\infty. \]
Then the function
\[ F(s):=\int_{x}^\infty f(y) y^{-s-1} dy  \]
will have
\[ \int_{T}^{2T} |F(\sigma+it)|^2 dt \ll (\log T) \int_x^\infty [\varphi(y)]^2 y^{-2\sigma-2} dy +\left(\int_x^{\infty} \varphi(y) y^{-\sigma-1} dy\right)^2 \]
for $\sigma\ge \sigma_0$ and $T\ge 2$, where the implicit constant depends only on the doubling constant $C$.
\end{lemma}}

\begin{proof}
By Fubini's theorem, we may interchange the order of integration and make the following direct computation: 
\begin{align}\nonumber  \int_{T}^{2T} |F(\sigma+it)|^2 dt & = \int_{x}^{\infty} \int_{x}^{\infty} \left(\int_{T}^{2T} \left(\frac{y}{z}\right)^{-it} dt\right) y^{-\sigma-1}z^{-\sigma-1} f(y)\overline{f(z)} dy dz \\
& \le  \int_{x}^{\infty}  \int_{x}^{\infty} \left|\int_{T}^{2T} \left(\frac{y}{z}\right)^{-it} dt\right| y^{-\sigma-1}z^{-\sigma-1} \varphi(y)\varphi(z) dydz . \label{eq:isplit} \end{align}
Now
\[  \int_{T}^{2T} \left(\frac{y}{z}\right)^{-it} dt\ll \begin{cases} T, & |\log(y/z)| \le 1/T, \\
            \frac{1}{|\log (y/z) |}, & |\log(y/z)| > 1/T. \end{cases} \]
 Consequently,
\begin{align*}  \int_{y/2\le z\le 2y} \left|\int_{T}^{2T} \left(\frac{y}{z}\right)^{-it} dt\right| y^{-\sigma-1}z^{-\sigma-1} \varphi(z) dz
& \ll \varphi(y) y^{-\sigma-1} \Big(1+\int_{1/T}^2 \frac{1}{\xi} d\xi \Big) \\
& \le  \varphi(y) y^{-\sigma-1} \big(1+\log T \big). \end{align*}
Since the inner integral in \eqref{eq:isplit} is uniformly bounded for $z\le y/2$ and $z\ge 2y$, we get the desired estimate by plugging the preceding bound into  \eqref{eq:isplit}.
\end{proof}

\vbox{\section{First step of the proof of Theorem~\ref{thm:main2}}\label{sec:constr}

\subsection{Construction of an Euler product whose meromorphic continuation vanishes at $s=\nu$}  \label{sec:euler}
The first step of the proof of Theorem~\ref{thm:main2} consists in picking a subsequence $\mathcal{P}_{\nu}$ of the primes such that the Euler product
\[ \zeta_{\mathcal{P}_{\nu}}(s):=\prod_{p\in \mathcal{P}_{\nu}} \frac{1}{(1-p^{-s})}\]
extends to a meromorphic function in $\mathbb{C}_{\nu/2}$ and $1/\zeta_{\mathcal{P}_{\nu}}(s)$ is analytic in $\mathbb{C}_{\nu/2}$ with only one zero, which has multiplicity $m_{\mathcal{Z}}(\nu)$ and is located at $s=\nu$. The desired Helson zeta function $\zeta_{\chi}(s)$ of Theorem~\ref{thm:main2} will then be of the form
\begin{equation} \label{eq:zetadec}  \zeta_{\chi}(s)=\frac{\zeta(s)}{\zeta_{\mathcal{P}_{\nu}}(s)} \prod_{p\in \mathcal{P}_{\nu}} \frac{1}{(1-\chi(p)p^{-s})}.\end{equation}
We will come back to the question of how to pick the numbers $\chi(p)$ for $p$ in $\mathcal{P}_{\nu}$ in the next section. One may however notice that the basic  ideas to be used to solve that problem appear, in a simpler form, in the construction to be carried out now. }

In the selection of the sequence $\mathcal{P}_{\nu}$, it will be essential to ensure that 
\begin{equation}\label{eq:toprovethm}   \limsup_{T\to \infty}\frac{1}{2T} \int_{-T}^{T} \left|\log \zeta_{\mathcal{P}_{\nu}}({\sigma+i t})-\log P_x \zeta_{\mathcal{P}_{\nu}}({\sigma+i t})\right|^2 dt 
\le C \sum_{p>x} p^{-2\sigma} \end{equation}
for some constant $C$, uniformly for $\sigma\ge \sigma_0$ whenever $\sigma_0>1/2$. It will become clear that the Dirichlet series 
\begin{equation} \label{eq:abs} \sum_{p\in \mathcal{P}_{\nu}} \sum_{\ell=2}^{\infty}   (\log p) p^{-\ell s} \end{equation}
will be absolutely convergent\footnote{As a matter of fact, to prove Theorem~\ref{thm:main2}, we only need the obvious fact that the series in \eqref{eq:abs} converges absolutely in $\mathbb{C}_{1/2}$, but it is of some independent interest to have meromorphic continuation to a larger half-plane.} in $\mathbb{C}_{\nu/2}$, and we may therefore restrict our attention to sums of the form
\[ D(s):=\sum_{p\in \mathcal{P}_{\nu}} p^{-s}. \] We will use the same symbol $D(s)$ for the meromorphic continuation of this function.
Setting $m=m_{\mathcal{Z}}(\nu)$, we require that
\[ D'(s)+\frac{m}{s-\nu} \]
defines an analytic function in $\mathbb{C}_{\nu/2}$. By integration and summation by parts, we may express this difference as
\begin{align}  \label{eq:dprim}  D'(s)+\frac{m}{s-\nu} & =-\sum_{p\in \mathcal{P}_{\nu}}(\log p) p^{-s} + m \int_{1}^{\infty} x^{\nu-s-1} dx \\
& = - s \int_{1}^{\infty}\left(\sum_{p\in \mathcal{P}_{\nu}, p\le x} \log p-m \frac{\left(x^{\nu}-1\right)}{\nu}  \right) x^{-s-1}dx , \nonumber \end{align}
where we initially assume that $\Real s>1$. We see, however, that the right-hand side will extend analytically to $\Real s > \nu/2$ if
\begin{equation}\label{eq:onezz}  \sum_{p\in \mathcal{P}_{\nu}, p\le x} \log p =\frac{m x^{\nu}}{\nu}+O\left(x^{\nu/2+\varepsilon}\right) \end{equation}
for every $\varepsilon>0$. Hence we wish to prove that \eqref{eq:onezz} holds when $\mathcal{P}_{\nu}$ has been suitably chosen. In addition, our aim will be to show that the function $D(s)$, which will then be analytic in the domain $\mathbb{C}_{\nu/2}\setminus (\nu/2, \nu]$, satisfies
\begin{equation} \label{eq:toprimes} \lim_{T\to\infty} \frac{1}{2T}\int_{-T}^T \left|D(\sigma+it)-\sum_{p\in \mathcal{P}_{\nu} , p\le x} p^{-\sigma-it}\right|^2 dt =\sum_{p\in \mathcal{P}_{\nu}, p>x} p^{-2\sigma}, \end{equation}
uniformly for $\sigma\ge \sigma_0$ whenever $\sigma_0>1/2$.

We make an inductive construction to find the sequence $\mathcal{P}_{\nu}$. Let $\theta$ be a parameter such that $\pi(x+x^{\theta})-\pi(x)\gg x^{\theta}/\log x$ holds for large $x$. By Lemma \ref{lem:gap}, we may choose $\theta=21/40$, but we prefer to carry out the construction keeping the numerical value of $\theta$ unspecified.  We define a sequence of real numbers $x_k$ inductively by requiring $x_{k+1}=x_k+x_k^{\theta}$, say, with $x_1:=2$. We will now describe in detail how we choose suitable primes in the interval $[x_k, x_{k+1})$. Our induction hypothesis is that
\begin{equation} \label{eq:pk0} \sum_{p\in \mathcal{P}_{\nu}, p\le x_k} \log p=\frac{m x_k^\nu}{\nu}+O(\log x_k). \end{equation}
We may agree that $2$ is in $\mathcal{P}_{\nu}$, but it does not really matter how we start the construction, since \eqref{eq:pk0} will hold trivially for small $k$ in any case. 
If we now pick a suitable number of primes $p$ in the interval $[x_k, x_{k+1})$ and declare these primes to constitute $\mathcal{P}_{\nu}\cap [x_k,x_{k+1})$, then we will be able to 
have also
\[ \sum_{p\in \mathcal{P}_{\nu}, p\le x_{k+1}} \log p=\frac{m x_{k+1}^\nu}{\nu}+O(\log x_{k+1}). \] 
In fact, we need $O\big(x_k^{\nu+\theta-1}/\log x_k\big)$ primes to achieve this, and by our assumption on $\theta$, there are $\gg x_k^{\theta}/\log x_k$ primes in this interval, whence this is feasible. By induction, we get in this way \eqref{eq:pk0} for all positive integers $k$.
It is then plain that for all positive numbers $x$, we will have
\begin{equation} \label{eq:basic0} \sum_{p\in \mathcal{P}_{\nu}, p\le x} \log p=\frac{m x^\nu}{\nu}+O(x^{\nu+\theta-1}), \end{equation}
independently of how the primes in any interval $[x_k, x_{k+1})$ are chosen. Since $1-\theta\ge 19/40 \ge \nu/2$, we have the desired bound for the remainder term in \eqref{eq:onezz}. It is also plain that
\[ \sum_{p\in \mathcal{P}_{\nu}} p^{-2\sigma} \ll \sum_{k=1}^{\infty} x_k^{-2\sigma} x_k^{\nu+\theta-1}/\log x_k \ll \sum_{j=1}^{\infty} 2^{(\nu-2\sigma)j} j^{-1} < \infty \]
whenever $\sigma>\nu/2$. Hence the Dirichlet series in \eqref{eq:abs} converges absolutely when $\Real s >\nu/2$, as anticipated above.

We are now left with a local problem, namely how to choose appropriately $O\big(x_k^{\nu+\theta-1}/\log x_k\big)$ primes in $[x_k, x_{k+1})$. Our goal is to do this so that the distance between consecutive primes is as large as possible. 

Since we will now consider just one interval, we set  $x_k=x$. The average distance between our primes in $[x, x+x^{\theta})$ will be of order $x^{1-\nu}\log x$; we wish to have a separation between our primes which is essentially of this magnitude. This we can achieve in the following way. We enumerate the available primes in $[x,x+x^{\theta})$: We arrange them by ascending magnitude and call them $p_j$ with $j=1,..., K$ and $K\gg x^\theta/\log x$. Choosing $j=\ell [c x^{1-\nu}]$ for a suitable constant $c$ and $1\le \ell \le K/[c x^{1-\nu}]$, we get a sufficient number of primes, ensuring  also that the distance between two consecutive primes is $\gg x^{1-\nu}$. We may express this important separation property of $\mathcal{P}_{\nu}$ as follows: There exists a positive constant $c$ such that
\begin{equation} \label{eq:sep} \inf_{p'\in \mathcal{P}_{\nu}, p'\neq p} |p-p'|\ge c p^{1-\nu} \end{equation}
for every $p$ in $\mathcal{P}_{\nu}$.

\subsection{Computation of  mean square distances}\label{sec:meansix}  Since  \eqref{eq:onezz} holds, the function $D(s)$ defined in the preceding subsection is analytic in $\mathbb{C}_{1/2}\setminus (1/2,\nu]$. We now prove that it in fact satisfies \eqref{eq:toprimes}. 

While \eqref{eq:dprim} gives us the desired meromorphic continuation of $D(s)$, it is not convenient for computing mean squares. We may rewrite \eqref{eq:dprim} to get it into a more manageable form, by applying summation and integration by parts only to the ``tails'' of respectively the Dirichlet series and the integral 
on the right-hand side of \eqref{eq:dprim}. This yields the decomposition
\begin{align} \label{eq:dps} D(s)=& \sum_{p\in \mathcal{P}_{\nu}, p\le x_k} p^{-s} + \int_{\sigma}^\infty \frac{mx_k^{\nu-u-it}}{(u+it-\nu)} du \\
 & +\int_{\sigma}^{\infty} (u+it) \int_{x_k}^{\infty}\left(\sum_{p\in \mathcal{P}_{\nu}, p\le x} \log p-\frac{m x^{\nu}}{\nu}  \right) x^{-u-1-it}dx du + O\left(x_k^{-\sigma}\log x_k \right), \nonumber \end{align}
where we used \eqref{eq:pk0} to get the remainder term. The exact cut-off value $x_k$ is chosen only for convenience, giving us one term less to account for. From \eqref{eq:dps} we then get
\begin{align} \label{eq:dss} D(s)-\sum_{p\in \mathcal{P}_{\nu}, p\le x} p^{-s}=&  \sum_{p\in \mathcal{P}_{\nu}, x<p\le x_k} p^{-s} + \int_{\sigma}^{\infty}\frac{m x_k^{\nu-u-it}}{(u+it-\nu)} du  \\
 & + \int_{\sigma}^{\infty} (u+it) \int_{x_k}^{\infty}\left(\sum_{p\in \mathcal{P}_{\nu}, p\le x} \log p-\frac{m x^{\nu}}{\nu}  \right) x^{-u-it-1}dxdu + O\left(x_k^{-\sigma} \right), \nonumber \end{align}
assuming that $x<x_k$ and $t\neq 0$. Our goal is now to compute the contribution from each of the terms in this composition and to choose a value for $x_k$ that yields an optimal balance between the respective bounds. For the first term on the right-hand side of \eqref{eq:dss}, we have
\begin{equation} \label{eq:first} \int_{-T}^T  \left|\sum_{p\in \mathcal{P}_{\nu}, x<p\le x_k} p^{-\sigma-it}\right|^2 = 2T\sum_{p\in \mathcal{P}_{\nu}, x<p\le x_k} p^{-2\sigma}
+ O\left(1+ x_k^{2\nu-2\sigma} \right) \end{equation}
by the separation property \eqref{eq:sep} of $\mathcal{P}_{\nu}$ along with Lemma~\ref{lem:secmom}. For the second term, we have trivially 
\begin{equation} \label{eq:second}  \int_{-T}^T \left|\int_{\sigma}^{\infty}\frac{m x_k^{\nu-u-it}}{(u+it-\nu)} du \right|^2 dt \ll 1+ x_k^{2\nu-2\sigma}/\log x_k. \end{equation}
Finally,  by the Cauchy--Schwarz inequality,
\begin{align} \nonumber
\Bigg|\int_{\sigma}^{\infty} (u+it) \int_{x_k}^{\infty}& \left(\sum_{p\in \mathcal{P}_{\nu}, p\le x} \log p-\frac{m x^{\nu}}{\nu}  \right) x^{-u-1-it}dx du\Bigg|^2 \\
& \le \frac{1}{\sigma} \int_{\sigma}^{\infty} |u+it|^2 u^2 \Bigg|\int_{x_k}^{\infty} \left(\sum_{p\in \mathcal{P}_{\nu}, p\le x} \log p-\frac{m x^{\nu}}{\nu}  \right) x^{-u-1-it}dx\Bigg|^2 du. \label{eq:intdy} \end{align}
By \eqref{eq:basic0} and Lemma~\ref{lem:Isec} applied with $\varphi(y)=C y^{\nu+\theta-1}$, we find that
\[  \int_{\Theta\le |t|\le 2\Theta} \left|\int_{x_k}^{\infty}\left(\sum_{p\in \mathcal{P}_{\nu}, p\le x} \log p-\frac{m x^{\nu}}{\nu}  \right) x^{-\sigma-1-it}dx\right|^2 dt
\ll  x_k^{2\nu+2\theta-2-2\sigma} \log \Theta \]
whenever $\Theta\ge 2$. Using Fubini's theorem and splitting the integral dyadically when integrating \eqref{eq:intdy} in the range $2\le |t| \le T$, this yields  
\begin{equation} \label{eq:third} \int_{-T}^T \Bigg|\int_{\sigma}^{\infty} (u+it) \int_{x_k}^{\infty}\left(\sum_{p\in \mathcal{P}_{\nu}, p\le x} \log p-\frac{m x^{\nu}}{\nu}  \right) x^{-u-1-it}dx du\Bigg|^2dt\ll 
x_k^{2\nu+2\theta-2-2\sigma} T^2 \log T.\end{equation}
Choosing $k$ such that $x_k\sim T^{1/(2\nu-1)}$ and taking into account \eqref{eq:first}, \eqref{eq:second}, and \eqref{eq:third}, we therefore get
\begin{align} \nonumber \frac{1}{2T} \left|\int_{-T}^{T}D(\sigma+it)-\sum_{p\in \mathcal{P}_{\nu}, p\le x} p^{-\sigma-it}\right|^2 dt &
= \sum_{p\in \mathcal{P}_{\nu}, x<p\le x_k} p^{-2\sigma}+ O\left(T^{-1}\right) \\
& +O\left(T^{-\frac{2\sigma-1}{2\nu-1}}\right)+ O\left(T^{\frac{4\nu+2\theta-3-2\sigma}{2\nu-1}} \log T\right), \label{eq:fromsix} \end{align}
which is uniform for $\sigma\ge \sigma_0$. We get the convergence for $\sigma>1/2$ required by Theorem~\ref{thm:vor2} if $\nu\le (2-\theta)/2$. This means that we should have
$\nu\le 59/80$ when $\theta=21/40$, in accordance with the condition on $\nu$ in Theorem~\ref{thm:main2}.

\section{Proof of Theorem~\ref{thm:main}, Theorem~\ref{thm:main3}, and Theorem~\ref{thm:main2}}\label{sec:main}

\subsection{Main lines of the proof}\label{sec:mainlines} The proof of each of the three theorems mentioned in the title of this section is essentially the same, and we therefore give a joint presentation of the proof. By assumption, the Bohr--Landau condition holds trivially in all three cases, so we only need to verify that condition \eqref{eq:logint} of Theorem~\ref{thm:vor2} holds. Our plan is first to present the common main lines of the proof and then, when we come to the technical details of each of the three cases, to treat them in parallel. In each step of the argument, we present first the case of Theorem~\ref{thm:main} and Theorem~\ref{thm:main2} and then we show how to modify it to cover the more manageable case of Theorem~\ref{thm:main3} as well.

In either case, we will construct a sequence of primes $\mathcal{P}$ and a sequence of unimodular numbers $\chi(p)$ so that the function
\[ \zeta_{\chi;\mathcal{P}}(s):=\prod_{p\in \mathcal{P}} \frac{1}{(1-\chi(p) p^{-s})} \]
extends meromorphically to $\mathbb{C}_{1/2}$ and has $Z\big(\zeta_{\chi;\mathcal{P}}(s)\big)\cap \mathbb{C}_{1/2}=\mathcal{Z}$. In the case of Theorem~\ref{thm:main2}, it is crucial that we choose $\mathcal{P}$ as a subsequence of the sequence $\mathcal{P}_{\nu}$ of the preceding section, so that our Helson zeta function takes the form 
\[ \zeta_{\chi}(s) :=\frac{\zeta(s)}{\zeta_{\mathcal{P}_{\nu}}(s)} \zeta_{\chi;\mathcal{P}}(s)\zeta_{\chi;\mathcal{P}_{\nu}\setminus \mathcal{P}}(s). \]
Here the values $\chi(p)$ for $p$ in $\mathcal{P}_{\nu}\setminus \mathcal{P}$ are chosen using our random model so that
\begin{equation} \label{eq:random} \lim_{T\to \infty} \frac{1}{2T}\int_{-T}^T\Big|\log \zeta_{\chi;\mathcal{P}_{\nu}\setminus \mathcal{P}}(s)-\log P_x\zeta_{\chi;\mathcal{P}_{\nu}\setminus \mathcal{P}}(s)
\Bigg|^2 dt = \sum_{p\in \mathbb{N}(\mathcal{P}_{\nu}\setminus \mathcal{P}), p>x} \sum_{\ell=1}^{\infty} \ell^{-2} p^{-2\ell\sigma},\end{equation}
uniformly for $\sigma\ge \sigma_0$ for every $\sigma_0>1/2$ and $x>1$. 
In the other two cases, instead of $\mathcal{P}_\nu\setminus \mathcal{P}$, we use the whole sequence of primes not belonging to $\mathcal{P}$ which we denote by $\mathcal{P'}$. Then our Helson zeta function is defined as
\[ \zeta_{\chi}(s) :=\zeta_{\chi;\mathcal{P}}(s)\zeta_{\chi;\mathcal{P'}}(s), \]
where $\zeta_{\chi;\mathcal{P'}}(s)$ is chosen by means of the same random model that was used in the case of Theorem~\ref{thm:main2}.

Most of the proof will consist in first constructing and then analyzing the function
\begin{equation} \label{eq:Ddef} D(s):=\sum_{p\in \mathcal{P}} \chi(p) p^{-s}, \end{equation} 
which is the ``essential part''  part of the Dirichlet series of $\log \zeta_{\chi;\mathcal{P}}$ in the precise sense that the remaining part $\log \zeta_{\chi;\mathcal{P}}(s) -D(s)$ is absolutely convergent in $\mathbb{C}_{1/2}$. Writing respectively
\begin{align} 
\log  \zeta_{\chi}(s)& =D(s) +\log \zeta(s)-\log \zeta_{\mathcal{P}_{\nu}}(s)+\log\zeta_{\chi;\mathcal{P}}(s)-D(s)+\log \zeta_{\chi;\mathcal{P}_{\nu}\setminus \mathcal{P}}(s) \label{eq:logone} \\ 
\log \zeta_{\chi}(s)& =D(s)+\log\zeta_{\chi;\mathcal{P}}(s)-D(s)+\log \zeta_{\chi;\mathcal{P'}}(s) \label{eq:logtwo}, \end{align}
we then check that
\begin{itemize}
\item[(A)] each of the terms $\log \zeta(s)$, $\log \zeta_{\mathcal{P}_{\nu}}(s)$, $\log\zeta_{\chi;\mathcal{P}}(s)-D(s)$, $\log \zeta_{\chi;\mathcal{P}_{\nu}\setminus \mathcal{P}}(s)$ in \eqref{eq:logone} can be approximated by the partial sum of its Dirichlet series in the mean square sense;
\item[(B)] each of the terms $\log\zeta_{\chi;\mathcal{P}}(s)-D(s)$, $\log \zeta_{\chi;\mathcal{P}'}(s)$ in \eqref{eq:logtwo} can be approximated by the partial sum of its Dirichlet series in the mean square sense.
\end{itemize}
In the case of (A), we use then respectively a well known property of $\log \zeta(s)$, \eqref{eq:fromsix} of the preceding section, the fact that $\log\zeta_{\chi;\mathcal{P}}(s)-D(s)$ has an absolutely convergent Dirichlet series, and \eqref{eq:random}. In the case of (B), we just use that $\log\zeta_{\chi;\mathcal{P}}(s)-D(s)$ has an absolutely convergent Dirichlet series and the identity for $\log \zeta_{\chi;\mathcal{P}'}(s)$ corresponding to its counterpart to \eqref{eq:random}.

We conclude that since $\log\zeta_{\chi;\mathcal{P}}(s)-D(s)$ is analytic in $\mathbb{C}_{1/2}$ independently of how $\chi$ is chosen, what remains to establish our three theorems is to pick $\mathcal{P}$ and corresponding values $\chi(p)$ for $p$ in $\mathcal{P}$ such that first, by Lemma~\ref{lem:exist}, $D'(s)$ is meromorphic in $\mathbb{C}_{1/2}$, with simple poles at each of the points $\rho$ of $\mathcal{Z}$ of residue $m_{\mathcal{Z}}(\rho)$, and
second, that
\begin{equation} \label{eq:smcond}  \lim_{T\to \infty} \frac{1}{2T} \int_{-T}^T \Bigg| D(\sigma+it)-\sum_{p\le x} \chi(p) p^{-\sigma-it}\Bigg|^2 dt =
\sum_{p>x} p^{-2\sigma}, \end{equation}
uniformly for $\sigma\ge \sigma_0$ whenever $\sigma_0>1/2$ and $x>1$, so that Theorem~\ref{thm:vor2} applies.
The next four subsections will accomplish these two tasks given $\mathcal{Z}$ satisfying either of the conditions of our three theorems.

We notice that from this point on, the work to be done is word for word the same for the two cases of Theorem~\ref{thm:main} and Theorem~\ref{thm:main2}. The only significant difference occurred above when we applied our random model respectively to $\mathcal{P}'$ and $\mathcal{P}_{\nu}\setminus \mathcal{P}$. This explains why we, with some extra work, could have considered meromorphic extensions of $\zeta_{\chi}(s)$ to $\mathbb{C}_{\nu/2}$ in the context of Theorem~\ref{thm:main2}, in contrast to what our methods permit us to do in relation to Theorem~\ref{thm:main}.

\subsection{Selection of $\mathcal{P}$ and $\chi(p)$ for $p$ in $\mathcal{P}$} It will be convenient to agree that $\alpha=1$ in the case of Theorem~\ref{thm:main3}. We begin by assuming that $\mathcal{Z}$ is a locally finite signed multiset\footnote{At this stage, it is convenient to demand a little less from the location of $\mathcal{Z}$ than what is required in Theorem~\ref{thm:main} and Theorem~\ref{thm:main3}.} in $\mathbb{C}_{\alpha/2}$ satisfying the conditions of either of our three theorems. 
We set
\begin{equation} \label{eq:R} R(s):=\sum_{\rho} m_{\mathcal{Z}}(\rho) \frac{y_{\rho}^{\rho-s}}{s-\rho}, \end{equation}
where $y_{\rho}$ are positive numbers to be chosen such that the series on the right-hand side converges absolutely in $\mathbb{C}_{\alpha/2}\setminus \mathcal{Z}$. Whatever choice we make for $y_{\rho}$, the virtue of $R(s)$ is that it has a simple pole at each point $\rho$ of $\mathcal{Z}$ of residue $m_{\mathcal{Z}}(\rho)$. We now make the specific choice 
\begin{equation} \label{eq:defyr} y_{\rho}:=(|\gamma|+1)^{1/(\alpha+\beta-1)} \end{equation}
for the proofs of Theorem~\ref{thm:main} and Theorem~\ref{thm:main2}, and set simply
\begin{equation} \label{eq:defyr2} y_{\rho}:=|\gamma| \end{equation}
when proving Theorem~\ref{thm:main3}, where as always $\rho=\beta+i\gamma$. These particular values for $y_{\rho}$ may at this point seem somewhat arbitrary, but it should become clear during the course of the proof that they are carefully tuned with the various requirements to be taken into account.

Our density condition on $\mathcal{Z}$ ensures the required convergence of the series in \eqref{eq:R}. Indeed,
in the first case, for $x\ge 2|t|$, we have
\begin{align*} \sum_{\rho: x\le |\gamma| < 2x} |m_{\mathcal{Z}}(\rho)| \frac{y_{\rho}^{\beta-\sigma}}{|\sigma+it-\rho |}
& \le 2\sum_{\rho: |\gamma|\le 2x}  |m_{\mathcal{Z}}(\rho)| x^{\frac{\beta-\sigma}{\alpha+\beta-1}-1} \\
&  = 2 x^{\frac{1-\sigma-\alpha}{\alpha-1/2}}N_{\mathcal{Z}}(1/2,2x) \\
& \quad +2 \int_{1/2}^{\alpha} (\log x) \frac{d}{du}\left(\frac{1-\sigma-\alpha}{\alpha+u-1}\right) x^{\frac{1-\sigma-\alpha}{\alpha+u-1}}N_{\mathcal{Z}}(u,2x) du \\
& \ll  x^{\frac{1/2-\sigma}{\alpha-1/2}} \log x,\end{align*}
where we in the last step used condition (c) of either of the two theorems. Hence summing over all dyadic intervals of the form $2^k x \le |\gamma| < 2^{k+1} x$, we get
\[ \sum_{\rho: |\gamma| \ge 2x} |m_{\mathcal{Z}}(\rho)| \frac{y_{\rho}^{\beta-\sigma}}{|\sigma+it-\rho |} 
\ll (\log x) x^{\frac{1/2-\sigma}{\alpha-1/2}}. \]
In the second case, on the conditions of Theorem~\ref{thm:main3}, we get\[ \sum_{\rho: x\le |\gamma| < 2x} |m_{\mathcal{Z}}(\rho)| \frac{y_{\rho}^{\beta-\sigma}}{|\sigma+it-\rho |}
\ll x^{-\sigma} (\log x)^{-1} , \]
again assuming $x\ge 2|t|$. We then conclude as in the first case by making another summation over dyadic intervals.

We next introduce the function
\[ q(x):=\sum_{\rho: y_{\rho}\le x} m_{\mathcal{Z}}(\rho) x^{\rho-1}, \]
which will play a crucial role in our construction. In the case of Theorem~\ref{thm:main} and Theorem~\ref{thm:main2}, we find that
\[ |q(x)|  \le \sum_{y_{\rho}\le x}|m_{\mathcal{Z}}(\rho)| x^{\beta-1}  \le  \sum_{\rho:|\gamma| \le x^{\alpha+\beta-1}} |m_{\mathcal{Z}}(\rho)| x^{\beta-1},
 \]
using the definition of $y_{\rho}$. Now, since
 \[ x^{\beta}= (1-1/e)^{-1} (\log x) \int_{\beta-1/\log x}^{\beta} x^{u} du, \]
 we infer from this that 
 \[ |q(x)| \le 
 (1-1/e)^{-1}(\log x) \int_{1/2-1/\log x}^{\alpha}  x^{u-1} N_{\mathcal{Z}}\left(u, e x^{\alpha+u-1}\right) du.  \]
Using finally condition (c) of either of the two theorems, we see that
\begin{equation}\label{eq:qbound} q(x)\ll x^{\alpha-1} \log x. \end{equation}
The corresponding computation on conditions (a) and (b) of Theorem~\ref{thm:main3} takes the simpler form
\begin{align} |q(x)| & \le \sum_{\rho: y_{\rho}\le x}|m_{\mathcal{Z}}(\rho)| x^{\beta-1}  =  \sum_{\rho: |\gamma| \le x} |m_{\mathcal{Z}}(\rho)| x^{\beta-1}  \nonumber \\
& \le (\log x)^{-\lambda} \sum_{\rho: |\gamma| \le x} |m_{\mathcal{Z}}(\rho)|  = (\log x)^{-\lambda} N_{\mathcal{Z}}(1/2, x) =O\left((\log x)^{-1-(\lambda-\kappa)} \right).
\label{eq:qbound2} \end{align}

We now make an inductive construction to find the sequence $\mathcal{P}$ suitable for the proof of either Theorem~\ref{thm:main} or Theorem~\ref{thm:main2}, by essentially the same argument that was used in Subsection~\ref{sec:euler}. It will in either case be convenient to pick $\mathcal{P}$ as a subsequence of $\mathcal{P}_{\nu}$ that was constructed in the preceding section.
We let $\theta$ and $x_k$ have the same meaning as before. Our induction hypothesis is now that
\begin{equation} \label{eq:pk1} \sum_{p\in \mathcal{P}, p\le x_k} \chi(p) \log p=\int_1^{x_k} q(y) dy +O(\log x_k). \end{equation}
If we now pick a suitable number of primes $p$ from $\mathcal{P}_{\nu}$ in the interval $[x_k, x_{k+1})$ and declare these primes to constitute $\mathcal{P}\cap [x_k,x_{k+1})$, then we will be able to 
have also
\[ \sum_{p\in \mathcal{P}, p\le x_{k+1}} \chi(p)  \log p=\int_1^{x_{k+1}} q(y) dy+O(\log x_{k+1}). \] 
In view of \eqref{eq:qbound}, we need $O\big(x_k^{\alpha+\theta-1}\big)$ primes to achieve this. By construction, there are $\gg x^{\nu+\theta-1}/\log x$ primes from $\mathcal{P}_{\nu}$ in $[x_{k},x_{k+1})$, so this is feasible. Hence, if
\[ \int_{1}^{x_{k+1}} q(y) dy - \sum_{p\in \mathcal{P}, p\le x_k} \chi(p) \log p = R_k e^{i c_k} , \quad 0<R_k \ll x^{\alpha+\theta-1},\]
we may therefore choose $\sim R_k /\log x_k$ primes $p$ in $[x_k, x_{k+1})$ to get precisely \eqref{eq:pk1}; we set $\chi(p)=e^{i c_k}$ for these primes and declare them to be $\mathcal{P}\cap [x_k,x_{k+1})$. Arguing as in the preceding section and using the separation property \eqref{eq:sep} of $\mathcal{P}_{\nu}$, we can do this in such a way that
\begin{equation} \label{eq:sep1} \inf_{p'\in \mathcal{P}, p'\neq p} |p-p'|\ge c p^{1-\alpha}/\log p \end{equation}
for every $p$ in $\mathcal{P}$.

By induction, we get in this way \eqref{eq:pk1} for all positive integers $k$.
It is then plain that for all positive numbers $x$, we will have
\begin{equation} \label{eq:basic} \sum_{p\in \mathcal{P}, p\le x}\chi(p) \log p=\int_1^x q(y) dy +O\left(x^{\alpha+\theta-1}\log x \right), \end{equation}
independently of how the primes in any interval $[x_k, x_{k+1})$ are chosen. 

On the assumptions of Theorem~\ref{thm:main3}, we act in the same way, but require instead $x_{k+1}=x_k+C(\log x_k)^{2+\varepsilon}$ for $0<\varepsilon<\lambda-\kappa$ and a constant $C$ which is so large that $\pi(x_{k+1})-\pi(x_k)\ge 1$ holds for all $k$. This is feasible because, by \eqref{eq:qbound2}, we need to pick at most one prime in each interval $[x_k, x_{k+1})$. The same construction then gives us a sequence of primes $\mathcal{P}$ and corresponding values $\chi(p)$ such that
\begin{equation} \label{eq:basic2} \sum_{p\in \mathcal{P}, p\le x}\chi(p)  \log p=\int_1^x q(y) dy +O\big(\log x \big). \end{equation}
 
\subsection{Construction and representation of $D(s)$} Our next goal is to analyze the  function $D(s)$ defined in \eqref{eq:Ddef} 
in a similar way as was done for its counterpart in the preceding section.
We first consider the meromorphic continuation of $D(s)$. We assume that $\Real s >1$ and use again partial summation to obtain
\[ D'(s)-R(s)=s\int_{1}^{\infty} \Bigg(\sum_{p\in \mathcal{P}, p\le x} \chi(p) \log p-\int_1^x q(y)dy \Bigg)  y^{-s-1} dy. \]
Here the right-hand side extends to an analytic function in either $\mathbb{C}_{\alpha+\theta-1}$ or $\mathbb{C}_0$, depending on whether 
\eqref{eq:basic} or \eqref{eq:basic2} holds. In the first case, we have $\alpha+\theta-1\le 59/80+21/40-1<1/2$. This means that in either case, the function 
$D'(s)-R(s)$ is analytic in $\mathbb{C}_{1/2}$. Since $R(s)$ has simple poles at each of the points $\rho$ of $\mathcal{Z}$ of residue $m_{\mathcal{Z}}(\rho)$, we conclude that $D'(s)$ has a meromorphic continuation of the desired type. As before, we continue to use the symbols $D(s)$ and $D'(s)$ for respectively the analytic and meromorphic continuation as well.
  
We are now left with the problem of verifying \eqref{eq:smcond}. To this end, we begin by making a suitable decomposition of $D(s)$, analogous to that found in \eqref{eq:dps}. By summation  by parts, we obtain first
\begin{equation} \label{eq:dprimes}  D'(s)=-\sum_{p\in \mathcal{P}, p\le x} \chi(p)(\log p) p^{-s}-s\int_{x}^{\infty} \Bigg(\sum_{p\in \mathcal{P}, x<p\le y} \chi(p)\log p\Bigg) (\log y)^{-1} y^{-s-1} dy. \end{equation}
We will again choose $x=x_k$ for some $k$, to save us one term to estimate.
With this choice, the integral on the right-hand side of \eqref{eq:dprimes} can be written as
\begin{align*} s \int_{x_k}^{\infty} \Bigg(&\sum_{p\in \mathcal{P},x_k<p\le y}  \chi(p)\log p\Bigg) y^{-s-1} dy =  s \int_{x_k}^{\infty} \Bigg(\sum_{p\in \mathcal{P},p\le y} \chi(p)\log p-\int_{1}^y q(z) dz\Bigg)y^{-s-1} dy \\ & +s \int_{x_k}^{\infty} \int_{x_k}^y q(z) dz y^{-s-1} dy+s\int_{x_k}^\infty \Bigg(\int_1^{x_k} q(z)dz-\sum_{p\in \mathcal{P}, p\le x_k} \chi(p)\log p
\Bigg)y^{-s-1} dy 
  \\
& = s \int_{x_k}^{\infty} \left(\sum_{p\in \mathcal{P}, p\le y} \chi(p)\log p-\int_{1}^y q(z) dz\right)y^{-s-1}dy + \sum_{y_{\rho}\le x_k}m_{\mathcal{Z}}(\rho) \frac{x_k^{\rho-s}}{s-\rho} \\
& \quad + \sum_{y_{\rho}> x_k} m_{\mathcal{Z}}(\rho) \frac{y_{\rho}^{\rho-s}}{s-\rho}+O\left(x_k^{-\sigma}\log x_k\right).\end{align*}
Hence combining \eqref{eq:dprimes} with the latter expression and assuming $x_k>x$ for some fixed $x$,  we obtain the decomposition
\begin{equation} \label{eq:Ddec} D(s)-\sum_{p\in \mathcal{P}, p\le x} \chi(p) (\log p) p^{-s} =F_k(s)+I_k(s)+S_k(s)+\Sigma_k(s)+O(x_k^{-\sigma}), \end{equation}
where 
\begin{align*} F_k(s)& :=\sum_{p\in \mathcal{P}, x<p\le x_k} \chi(p) p^{-s} \\
I_k(s)& := \int_s^{\infty+it} w \int_{x_k}^{\infty} \Bigg(\sum_{p\in \mathcal{P}, p\le y} \chi(p)\log p-\int_{1}^y q(z) dz\Bigg)y^{-w-1}dy dw \\
S_k(s)&:= \int_{s}^{\infty+it} \left(\sum_{y_{\rho}\le x_k}m_{\mathcal{Z}}(\rho) \frac{x_k^{\rho-w}}{w-\rho}\right) dw \\
\Sigma_k(s) &:= \int_{s}^{\infty+it} \left(\sum_{y_{\rho}> x_k} m_{\mathcal{Z}}(\rho) \frac{y_{\rho}^{\rho-s}}{w-\rho}\right) dw, \end{align*}
where as usual $t= \Imag s$.
It is natural to group the computations into two subsections, one dealing with $F_k(s)$ and $I_k(s)$ and the other with the two sums $S_k(s)$ and $\Sigma_k(s)$ over the zeros and the poles. 

\subsection{Computation of the mean squares of $F_k(s)$ and $I_k(s)$}  We will now show that the mean squares of $F_k(s)$ are in accordance with \eqref{eq:logint} and that the mean squares of $I_k(s)$ are uniformly $O(T^{1-\varepsilon})$ in every half-plane $\sigma\ge \sigma_0>1/2$, in both cases on the condition that $x_k$ is suitably chosen.

\subsubsection{The case of  Theorem~\ref{thm:main} and Theorem~\ref{thm:main2}} By the separation condition \eqref{eq:sep1} and Lemma~\ref{lem:secmom}, we have
\begin{equation} \label{eq:Fk} \int_{-T}^{T} \big| F_k(\sigma+it)\big|^2 dt =2T \sum_{p\in \mathcal{P}, x<p\le x_k} p^{-2\sigma} 
+ O\left(x_k^{2\alpha-2\sigma}\log x_k \right) , \end{equation}
uniformly for $\sigma\ge 1/2$. 
If we  require that $x_k\ll T^{1/(2\alpha-1)}$, then the remainder term in \eqref{eq:Fk} is uniformly $O(T^{1-\varepsilon})$ when $\sigma\ge \sigma_0>1/2$. 

Now applying the Cauchy--Schwarz inequality as in Subsection~\ref{sec:meansix}, we get
\[ |I_k(s)|^2 \le \frac{1}{\sigma} \int_{\sigma}^{\infty} |u+it|^2 u^2 \Bigg|\int_{x_k}^{\infty} \Big(\sum_{p\in \mathcal{P}, p\le y} \chi(p) \log p-\int_{1}^y q(z) dz\Bigg)y^{-u-1-it}dy\Bigg|^2 du. \]
Hence, by Fubini's theorem and Lemma~\ref{lem:Isec}, applied with $\varphi(y)=cy^{\alpha+\theta-1}\log y$, we obtain 
\begin{align*} \int_{\Theta\le |t|\le 2\Theta} |I_k(\sigma+it)|^2 dt\  & \ll \ \frac{1}{\sigma} \int_{\sigma}^{\infty} |u+i\Theta|^2 u^2 x_k^{2(\alpha+\theta-1)-2u} \big((\log \Theta) \log u + 
(\log u)^2\big) du \\
&  \ll \Theta^2 \big(\log \Theta+\log x_k\big) x_k^{2(\alpha+\theta-1)-2\sigma} . \end{align*}
Adding this estimate dyadically, we find that
\[ \int_{-T}^T |I_k(\sigma+it)|^2 dt \ll T^2 \big(\log T+\log x_k\big) x_k^{2(\alpha+\theta-1)-2\sigma}, \]
which is of admissible size if $x_k \ge T^{1/(3-2\alpha-2\theta)}$. Combining our two restrictions on $\alpha$, we see that
\[ 3 - 2\alpha -2\theta \ge 2\alpha-1 ,\]
which yields $\alpha\le (2-\theta)/2$, in accordance with condition (a) of Theorem~\ref{thm:main} by the specific choice $\theta=21/40$, made in accordance with Lemma~\ref{lem:gap}. We could have chosen $k$ such $x_k$ would depend on $\sigma$, but this is of no special use, and we will in the sequel assume that $x_k\sim T^{1/(2\alpha-1)}$.

\subsubsection{The case of Theorem~\ref{thm:main3}} By Lemma~\ref{lem:secmom}, we get trivially  
\[ \int_{-T}^{T} \left| F_k(\sigma+it) \right|^2 dt =2T \sum_{p\in \mathcal{P}, x< p\le x_k} p^{-2\sigma} 
+ O\left(x_k^{2-2\sigma} \right), \]
and hence the remainder term is of appropriate size if $x_k\ll T$. Furthermore, arguing as in the preceding case, using \eqref{eq:basic2} and Lemma~\ref{lem:Isec}, now with $\varphi(y)=c \log y$, we get
\[   \int_{-T}^{T} |I_k(\sigma+it)|^2 dt\  \ll \  T^2 \big(\log T +\log x_k\big) x_k^{-2\sigma}.  \]
Choosing $k$ such that $x_k\sim T$, we see that the integral is $O(T^{2-2\sigma_0}\log T)$, which gives the desired behavior.
 
\subsection{Bounds for the sums over the zeros and the poles}\label{sec:zerospoles}
We will start from the trivial bounds
\begin{align}
\label{eq:Sk} \left|S_k(s)\right|&\le \int_{\sigma}^{\infty+it} \left(\sum_{\rho: y_{\rho} \le x_k}|m_{\mathcal{Z}}(\rho)| \frac{x_k^{\beta-u}}{|u+it-\rho|}\right) du \\
\label{eq:Sik} \left|\Sigma_k(s)\right| &\le \int_{\sigma}^{\infty+it} \left(\sum_{\rho: y_{\rho}>x_k} |m_{\mathcal{Z}}(\rho)| \frac{y_{\rho}^{\beta-u}}{|u+it-\rho|}\right) du, \end{align}
where we assume that $t$ is not an ordinate of any of the points $\rho=\beta+i\gamma$. 
What remains to establish the mean square condition \eqref{eq:logint}, is to show that 
\begin{equation} \label{eq:msbl} \int_{-T}^{T} |S_k(\sigma+it)|^2dt \ll T^{1-\varepsilon} \quad \text{and} \quad \int_{-T}^{T} |\Sigma_k(\sigma+it)|^2dt \ll T^{1-\varepsilon} \end{equation}
for some $\varepsilon>0$, uniformly when $\sigma\ge \sigma_0>1/2$. 
\subsubsection{The case of Theorem~\ref{thm:main} and Theorem~\ref{thm:main2}} To begin with, we observe that the inequality $y_{\rho} \le x_k$ is equivalent to
$1+|\gamma|\le T^{(\alpha+\beta-1)/(2\alpha-1)}$, and hence
\[ |S_k(s)|\le S_{k,T}(s)  := \int_{\sigma}^{\infty} \sum_{\rho:  |\gamma| \le T} |m_{\mathcal{Z}}(\rho)| \frac{x_k^{\beta-u}}{|u+it-\rho|}du. \]
We wish to restrict the summation on the right-hand side of  \eqref{eq:Sik} similarly, to the range 
$ |\gamma| \le 2T$. To this end, we find that
\begin{align} \label{eq:split2} \sum_{\rho}|m_{\mathcal{Z}}(\rho)| \frac{y_\rho^{\beta-u}}{|u+it-\rho|}
\ll & \sum_{\rho: |\gamma| \le 2T} |m_{\mathcal{Z}}(\rho)| \frac{y_{\rho}^{\beta-u}}{|u+it-\rho|} \\ 
& +  \sum_{\rho:  |\gamma|>2T}
|m_{\mathcal{Z}}(\rho)| \frac{y_{\rho} ^{\beta-u}}{|\rho|}. \nonumber \end{align}
The contribution from the latter term to the right-hand side of  \eqref{eq:Sik} is
\[ \ll \int_{\sigma}^\infty \left(\sum_{\rho: |\gamma|>2T}
|m_{\mathcal{Z}}(\rho)| \frac{y_{\rho} ^{\beta-u}}{(T+|\rho|)}\right)du \ll \sum_{\rho}
|m_{\mathcal{Z}}(\rho)| \frac{y_{\rho} ^{\beta-\sigma}}{|\rho|}. \]
We consider dyadic blocks and use that
\[ \sum_{\rho: X/2<|\gamma| \le X }|m_{\mathcal{Z}}(\rho)| y_{\rho}^{\beta-\sigma}  \ll 
\sum_{\rho: |\gamma| \le X }|m_{\mathcal{Z}}(\rho)| X^{\frac{\beta-\sigma}{\alpha+\beta-1}} \]
which holds because $y_\rho=(1+|\gamma|)^{1/(\alpha+\beta-1)}$.  Here the right-hand side can be written as 
\begin{align*} \sum_{\rho: |\gamma| \le X }|m_{\mathcal{Z}}(\rho)| X^{\frac{\beta-\sigma}{\alpha+\beta-1}}  
&  = X^{1/2-\sigma} N_{\mathcal{Z}}(1/2,X) +\log X \int_{1/2}^{\alpha} \frac{d}{du}\Bigg(\frac{u-\sigma}{\alpha+u-1}\Bigg)
X^{\frac{u-\sigma}{\alpha+u-1}} N(u,X) du. \end{align*}
Hence, using again condition (c) of either of the two theorems, we infer from this that
\begin{equation} \label{eq:extra2} \sum_{\rho: X/2<|\gamma| \le X }
|m_{\mathcal{Z}}(\rho)| y_{\rho}^{\beta-\sigma} \ll (\log X) X^{\frac{\alpha-\sigma}{\alpha-1/2}} .\end{equation}
Summing dyadically and using \eqref{eq:extra2}, we get a total contribution which is $O\big((\log T) T^{-\frac{\sigma-1/2}{\alpha-1/2}}\big)$.  
Hence, returning to \eqref{eq:split2}, we get
\[ \Sigma_k(s)
\ll  \int_{\sigma}^{\infty}  \sum_{\rho:  |\gamma| \le 2T} |m_{\mathcal{Z}}(\rho)| \frac{x_k^{\beta-u}}{|u+it-\rho|}du + O\left((\log T)T^{-\frac{\sigma-1/2}{\alpha-1/2}}\right), \]
and we are left with the problem of computing the integrals for 
\[ \Sigma_{k,T}(s)  :=  \int_{\sigma}^{\infty} \sum_{\rho:  |\gamma| \le 2T} |m_{\mathcal{Z}}(\rho)| \frac{y_{\rho}^{\beta-u}}{|u+it-\rho|}du .\]

We now turn to the integrals for the two functions $S_{k,T}(s)$ and $\Sigma_{k,T}(s)$. 
By Fubini's theorem and a direct computation, we get
\begin{align}\nonumber  \int_{-T}^{T} |S_{k,T}(\sigma+it)|^2 dt &  \ \le \int_{\sigma}^{\infty} \int_{\sigma}^{\infty} \sum_{\rho, \rho': T/2 \le \gamma , \gamma' \le 3T} 
\int_{-T}^{T} \frac{|m_{\mathcal{Z}}(\rho)| |m_{\mathcal{Z}}(\rho')| x_k^{\beta-u} x_k^{\beta'-u'}}{|u+it-\rho| |u'+it-\rho'|} dt du du'  \\
& \ll  \sum_{\rho, \rho':   |\gamma| , |\gamma'| \le T} 
 \frac{|m_{\mathcal{Z}}(\rho)| |m_{\mathcal{Z}}(\rho')| x_k^{\beta-\sigma}x_k^{\beta'-\sigma}}{(|\rho-\rho'|+1)} \log(e+|\rho-\rho'|). \label{eq:firstth}\end{align}
We apply the Cauchy--Schwarz inequality and the assumption that $|m_{\mathcal{Z}}(\rho)|\ll T^{\varepsilon}$ to infer from this that
\[ \int_{-T}^{T} |S_{k,T}(\sigma+it)|^2 dt \ll T^{\varepsilon} \sum_{\rho, \rho':  |\gamma|, |\gamma'| \le T}|m_{\mathcal{Z}}(\rho)| \frac{x_k^{2(\beta-\sigma)}}{\left(|\rho-\rho'|+1\right)}, \]
where the term $\log(e+|\rho-\rho'|)$ has been absorbed in the factor $T^{\varepsilon}$. Summing over $\rho'$ and using condition (b) of Theorem~\ref{thm:main}, we then get
\[  \int_{-T}^{T} |S_{k,T}(\sigma+it|^2 dt \ll T^{2\varepsilon} \sum_{\rho: |\gamma| \le T} |m_{\mathcal{Z}}(\rho)| x_k^{2(\beta-\sigma)}.  \]
Observing that
\[ \sum_{\rho: |\gamma| \le T} |m_{\mathcal{Z}}(\rho)| x_k^{2(\beta-\sigma)}= N_{\mathcal{Z}}(1/2,T) x_k^{2(1/2-\sigma)}
+(\log x_k)\int_{1/2}^{\alpha} x_k^{2(u-\sigma)} N_{\mathcal{Z}}(u,T) du \]  
and using that $x_k\sim T^{1/(2\alpha-1)}$ and condition (c) of either of the two theorems, we get
\[ \int_{-T}^{T} |S_{k,T}(\sigma+it|^2 dt \ll T^{2\varepsilon}\cdot T^{1-\frac{2\sigma-1}{2\alpha-1}}. \]
Since $\varepsilon$ can be chosen arbitrarily small, the desired estimate for $S_k(s)$ in \eqref{eq:msbl} follows. 

Finally, acting in exactly the same way, we get
\[ \int_{-T}^{T} |\Sigma_{k,T}(\sigma+it)|^2 dt \ll T^{2\varepsilon}  \sum_{\rho: |\gamma| \le 2T} |m_{\mathcal{Z}}(\rho)| y_{\rho}^{2(\beta-\sigma)} .\]
We now use that $y_{\rho}\sim |\gamma|^{1/(\alpha+\beta-1)}$, whence  
\[ \sum_{\rho: X<|\gamma|  \le 2X} |m_{\mathcal{Z}}(\rho)| y_{\rho}^{2(\beta-\sigma)} 
\ll N_{\mathcal{Z}}(1/2,2X) X^{\frac{2(1/2-\sigma)}{\alpha-1/2}}
+(\log x_k)\int_{1/2}^{\alpha} X^{\frac{2(u-\sigma)}{\alpha+u-1}} N_{\mathcal{Z}}(u,2X) du. \] 
Summing dyadically, we  conclude that
\[ \int_{-T}^{T} |\Sigma_{k,T}(\sigma+it)|^2 dt \ll T^{2\varepsilon}\cdot  T^{1-\frac{2\sigma-1}{2\alpha-1}},\]
which implies that the desired estimate for $\Sigma_k(s)$ in \eqref{eq:msbl} holds as well.

\subsubsection{The case of Theorem~\ref{thm:main3}}
The required estimates are again straightforward in this case. From \eqref{eq:Sk}, we get that
\[ S_k(s)\ll \int_{\sigma}^\infty \left(\sum_{\rho:  |\gamma| \le 2T} |m_{\mathcal{Z}}(\rho)| \frac{x_k^{\beta-u}}{|u+it-\rho|}\right) du, \]
since $y_{\rho}=|\gamma|$ and $x_k\sim T$. It follows that
\[ \int_{-T}^{T} |S_{k}(\sigma+it)|^2 dt    \ll  \sum_{\rho, \rho':  |\gamma| , |\gamma'|\le 2T} 
 \frac{|m_{\mathcal{Z}}(\rho)| |m_{\mathcal{Z}}(\rho')| T^{2-2\sigma}(\log T)^{-2\lambda}}{(|\rho-\rho'|+1)} \log(e+|\rho-\rho'|). \]    
Summing trivially and using condition (b) of Theorem~\ref{thm:main3}, we then get
 \[ \int_{-T}^{T} |S_{k}(\sigma+it)|^2 dt \ll  T^{2-2\sigma}, \]
 which yields the desired bound in \eqref{eq:msbl}.
 
 Finally, to deal with $\Sigma_k(s)$, we find that
 \[ \Sigma_k(s)
\ll  \int_{\sigma}^{\infty} \left( \sum_{\rho:  |\gamma| \le 2T} |m_{\mathcal{Z}}(\rho)| \frac{y_{\rho}^{\beta-u}}{|u+it-\rho|}\right)du + O\left((\log T)T^{-\sigma}\right) \]
when $|t|\le T$.
Since $y_{\rho}\le 2T$ in the sum on the right-hand side, the computation we just made covers this case as well.

 \section{Proof of Theorem~\ref{thm:nonuniversal}} \label{sec:proofany}
\subsection{General scheme of the construction of $\zeta_{\chi}(s)$ with $Z^+(\zeta_{\chi}(s))=\mathcal{Z}^+$} Throughout this section, we assume that $\mathcal{Z}^+$ is a locally finite multiset in $\mathbb{C}_{1/2}$ that is confined to either of the two strips  
$1/2<\Real s < 1$ or $1/2<\Real s \le 39/40$, depending on whether we take the truth of the Riemann hypothesis for granted or not. 

We begin by making the following dyadic partition of the strip $1/2 < \Real s < 1$. We set 
\begin{align*} U_j &:= \left\{s=\sigma+it: \ 1/2+2^{-(j+1)}\le \sigma < 1/2+2^{-j}, \ |t|\le 2^j \right\} \\
                      V_j &:= \left\{s=\sigma+it: \ 1/2+2^{-(j+1)}\le \sigma < 1, \ 2^j<|t|\le 2^{j+1} \right\} 
                        \end{align*}
and observe that
\[ \bigcup_{j=1}^{\infty} \left(U_j \cup V_j \right)= \big\{s: \ 1/2 < \Real s < 1\big\}. \]
We will, during the course of the proof,  assign to each $\rho$ in $\mathcal{Z}^+\cap \big(\bigcup_{j=1}^\infty V_j\big)$ a nearby point $\rho'$. It will be a tacit assumption that the points $\rho'$ are distinct and not contained in $\mathcal{Z}^+$. We assign the multiplicity $-m_{\mathcal{Z}^+}(\rho)$ to $\rho'$ and let $\mathcal{Z}$ denote the signed multiset obtained by adjoining the extraneous points $\rho'$ to $\mathcal{Z}^+$.
With each integer $j$, we associate two numbers $u_j, v_j \ge 1$ so that the following function is well defined in $\mathbb{C}_{1/2}$: 
\begin{equation} \label{eq:RR} R(s):=
\sum_{j=1}^\infty \Bigg(\sum_{\rho\in U_j} m_{\mathcal{Z}^+}(\rho) \frac{u_{j}^{\rho-s}}{s-\rho} 
+\sum_{\rho\in V_j}  m_{\mathcal{Z}^+}(\rho)  \Bigg(\frac{v_{j}^{\rho-s}}{s-\rho}-\frac{v_{j}^{\rho'-s}}{s-\rho'}
 \Bigg)\Bigg). \end{equation}
We will make the specific choice $v_j:=2^j$. It is clear that if we choose $u_j$ sufficiently large and $\rho'$ sufficiently close to $\rho$, then the series over $j$ will converge absolutely when $s$ is in $\mathbb{C}_{1/2} \setminus \mathcal{Z}$ and uniformly on compact subsets of this domain. Consequently, \eqref{eq:RR} will then define a meromorphic function in $\mathbb{C}_{1/2}$ with simple poles at the points $\rho$ of $\mathcal{Z}^{+}$ as well as at the ``extraneous'' points $\rho'$. Moreover, the residue of the pole at $\rho$ is $m_{\mathcal{Z}^+}(\rho)$, and the residue of the pole at $\rho'$ is $-m_{\mathcal{Z}^+}(\rho)$. Hence if we can find a sequence of primes $\mathcal{P}$ and unimodular numbers $\chi(p)$ so that
\begin{equation} \label{eq:Rp}  R(s)+\sum_{p\in \mathcal{P}} \chi(p) (\log p) p^{-s} \end{equation}
defines an analytic function in $\mathbb{C}_{1/2}$, then 
$Z^+\big(\zeta_{\chi;\mathcal{P}}(s)\big)\cap \mathbb{C}_{1/2}=\mathcal{Z}^+$ by the same reasoning as in Subsection~\ref{sec:mainlines}.
Having accomplished this, we use again Lemma~\ref{lem:almost} to construct $\chi(p)$ for $p$ not in $\mathcal{P}$. To finish the proof of Theorem~\ref{thm:nonuniversal}, it will therefore be enough to make
the function in \eqref{eq:Rp} analytic in $\mathbb{C}_{1/2}$.
 
We begin by setting $u_1:=1$ and define inductively 
\begin{equation} \label{eq:uj} u_j:= \exp\Big(j+\sum_{\rho\in U_j} |m_{\mathcal{Z}^+}(\rho)|  \Big) \end{equation}
for $j>1$.  For a fixed $s=\sigma+it$ in $\mathbb{C}_{1/2} \setminus \mathcal{Z}$, we choose $m$ such that $2^{-m+1}<\sigma-1/2$ and set
$\varepsilon:=\sigma-1/2 -2^{-m}$. Then
\[ \sum_{j=m}^\infty \left|\sum_{\rho\in U_j} |m_{\mathcal{Z}^+}(\rho)| \frac{u_{j}^{\rho-s}}{s-\rho}\right| 
\ll \sum_{j=m}^\infty  u_j^{-\varepsilon}
\log u_j <\infty \] 
by the exponential decay of $u_j$. We require next
\[ \delta_j\le  \big[\sum_{\rho\in V_j} m_{\mathcal{Z}^+}(\rho)\big]^{-1}. \] 
Hence if $m$ so large that $2^{m} e \ge |s| $, 
then we also get
\[\sum_{j=m}^\infty \sum_{\rho\in V_j} \Bigg| m_{\mathcal{Z}^+}(\rho)  \Bigg(\frac{v_{j}^{\rho-s}}{s-\rho}-\frac{v_{j}^{\rho'-s}}{s-\rho'}
\Bigg) \Bigg| \ll \sum_{j=m}^{\infty}  j \delta_j 2^{-\sigma j} \sum_{\rho \in V_j}  \le \sum_{j=m}^{\infty} j 2^{-\sigma j} <\infty. \]
We conclude that our requirements for $u_j$ and $\delta_j$ entail the desired uniform convergence on compact subsets of
$\mathbb{C}_{1/2}\setminus \mathcal{Z}$.

We now set 
\begin{equation}\label{eq:qdef}  q(x):= \sum_{j: u_j \le x} \sum_{\rho\in U_j} m_{\mathcal{Z}}(\rho) x^{\rho-1}+ \sum_{j: 2^j\le x} \sum_{\rho\in V_j} m_{\mathcal{Z}}(\rho) \big(x^{\rho-1}-x^{\rho'-1}\big). \end{equation}
We may then write
\begin{equation} R(s)=\int_{1}^\infty q(x) x^{-s}  dx   \label{eq:RHS} \end{equation}
when $\Re s>1$.
 Hence our task is to pick $\mathcal{P}$ and $\chi(p)$ such that 
       \[ R(s)+\sum_{p\in \mathcal{P}} \chi(p) (\log p) p^{-s} \]
       has an analytic continuation to $\mathbb{C}_{1/2}$. 
       We use once again partial integration and summation to deduce that     
\begin{equation} \label{eq:QP} R(s)+\sum_{p\in \mathcal{P}} \chi(p)(\log p) p^{-s}=  s \int_1^\infty \left(\int_1^x q(y) dy -
       \sum_{p\in \mathcal{P}, p\le x} \chi(p) \log p \right) x^{-s-1} dx. \end{equation}
As in the preceding section, we will make an inductive construction to ensure that 
\begin{equation} \label{eq:qp} \int_1^x q(y) dy + \sum_{p\in \mathcal{P}, p\le x} \chi(p) \log p \ll x^{1/2+\varepsilon} \end{equation}
for every $\varepsilon>0$,
which will yield absolute convergence in $\mathbb{C}_{1/2}$ of the integral on the right-hand side of \eqref{eq:QP}.

The construction differs slightly depending on whether we assume the Riemann hypothesis to be true or not. In either case, it will be convenient to require
\[ \rho'=\beta'+i\gamma, \quad \beta'<\beta, \]
which implies that
\begin{equation}\label{eq:require}
\big| x^{\rho}-x^{\rho'} \big| \le x^{\beta} \big|\beta-\beta'\big| \log x.
\end{equation}
\subsection{Construction of $\mathcal{P}$ and $\chi(p)$ in the unconditional case} We start from \eqref{eq:qdef}.  Recalling that $\Real \rho \le 39/40$ for all $\rho$ and \eqref{eq:require} holds, we find that
\begin{equation} \label{eq:qsum}  |q(x) |\le x^{-1/40} \sum_{j: u_j \le x}  \log u_j + x^{-1/40} 
(\log x) \sum_{j: 2^j \le x} \delta_j \sum_{\rho\in V_j} m_{\mathcal{Z}^+}(\rho)  \ll x^{-1/40} (\log x)^2,\end{equation}
since each of the sums in \eqref{eq:qsum} contains at most $(\log x)/\log 2$ terms. 

We now set $x_1:=2$ and define inductively  
$x_{k+1}:=x_k+x_k^{21/40}$ for $k\ge 1$. We will now describe in detail how to choose suitable primes in the interval $[x_k, x_{k+1})$. We make the induction hypothesis that
\begin{equation} \label{eq:pk} \sum_{p\in \mathcal{P}, p\le x_k} \chi(p) \log p=\int_1^{x_k} q(y) dy +O(\log x_k) \end{equation}
holds for all $k$. This is trivially true for $k=1$ independently of how we start the construction. We may for example choose to include $p=2$ in $\mathcal{P}$ and set $\chi(2)=1$. Our aim is now to pick a suitable number of primes $p$ in the interval $[x_k, x_{k+1})$ and declare these primes to constitute $\mathcal{P}\cap [x_k,x_{k+1})$, to 
have also
\begin{equation} \label{eq:step} \sum_{p\in \mathcal{P}, p\le x_{k+1}} \chi(p)  \log p=\int_1^{x_{k+1}} q(y) dy+O(\log x_{k+1}). \end{equation} 
By \eqref{eq:qsum} and the definition of the sequence $x_k$, we have
\[ \int_{x_k}^{x_{k+1}} q(y) dy \ll x_k^{21/40-1/40}(\log x)^2= x_k^{1/2}(\log x)^2. \]
Hence we need $O\big(x_k^{1/2} \log x_k\big)$ primes to achieve \eqref{eq:step}. According to Lemma~\ref{lem:gap}, there are $\gg x_k^{21/40} \log x_k$ primes in 
$[x_{k},x_{k+1})$. Setting
\[ \int_{1}^{x_{k+1}} q(y) dy -\sum_{p\in \mathcal{P}, p\le x_k} \chi(p) \log p= R_k e^{i c_k} , \quad 0<R_k \ll x_k^{1/2} ,\]
we may therefore choose $\sim R_k \log x_k$ primes $p$ in $[x_k, x_{k+1})$ to get precisely \eqref{eq:step}; we set $\chi(p)=e^{i c_k}$ for these primes and declare them to be 
$\mathcal{P}\cap [x_k,x_{k+1})$.

By induction, we get in this way \eqref{eq:pk} for all positive integers $k$.
It is then plain that for all positive numbers $x$, we will have
\[ \sum_{p\in \mathcal{P}, p\le x}\chi(p)  \log p=\int_1^x q(y) dy +O\left(x^{1/2}(\log x)^2\right) \]
which means that \eqref{eq:qp} holds.

\subsection{Construction of $\mathcal{P}$ and $\chi(p)$ in the conditional case} It will 
suffice to have $q(x)=o(1)$ because then
\[ \int_{x}^{x+2x^{1/2} \log x} q(y) dy =o\left(x^{1/2} \log x \right), \]
and we may use \eqref{eq:dudek} of Lemma~\ref{lem:gap} and act as in the preceding case, now with
\[ x_{k+1}:=x_k+2x_k^{1/2} \log x_k. \] 
This will give us
\[ \sum_{p\in \mathcal{P}, p\le x}\chi(p)  \log p =\int_1^x q(y) dy +o\left(x^{1/2} \log x \right), \]
and hence \eqref{eq:qp} will again hold.

To achieve the required asymptotics $q(x)=o(1)$, we define a new sequence
\[ \beta_0:=\max\left\{ \Re \rho: \ \rho\in U_1 \right\} \quad \text{and} \quad \beta_j:=\max \left\{ \Re \rho: \ \rho\in V_j \right\}, \ j\ge 1. \]
Starting again from \eqref{eq:qdef} and using \eqref{eq:require}, we then get
\[ |q(x)|  \le  x^{\beta_0-1} \sum_{j: u_j \le x}  \log u_j + (\log x) \sum_{j: v_j\le x} x^{\beta_j-1} \delta_j \sum_{\rho \in V_j} |m_{\mathcal{Z}}(\rho)|. \] 
Hence requiring 
\[ \delta_j \le (\beta_j-1) \big[j^2\sum_{\rho\in V_j} m_{\mathcal{Z}^+}(\rho)\big]^{-1}, \]
we infer that 
\[ |q(x)|  \le  x^{\beta_0-1} (\log x)^2 + \sum_{j=1}^{\infty} [(\beta_j-1)\log x] x^{\beta_j-1} j^{-2}=o(1), \] 
where we in the last step used that 
\[ \big[(\beta_j-1)\log x\big] x^{\beta_j-1} \le e^{-1} \quad \text{and} \quad \lim_{x\to \infty} \big[(\beta_j-1)\log x\big] x^{\beta_j-1} = 0. \]

 \section*{Acknowledgements} I am indebted to \'{E}ric Sa{\"i}as for introducing me to Bohr's approach to the Riemann hypothesis, for many inspiring discussions, and for valuable and constructive feedback on preliminary versions of this manuscript. I would also like to thank Andriy Bondarenko, Danylo Radchenko, Eero Saksman, Christian Skau, and Michel Weber for some helpful comments. Finally, I am indebted to the anonymous referee for a thorough review that helped me remove several inaccuracies from the text.



\end{document}